\documentclass[a4paper,11pt]{amsart}
\usepackage{amssymb,amsfonts,amsxtra,fullpage,
mathrsfs,placeins,graphicx,verbatim,stmaryrd,cite,color,lmodern,hyperref}
\usepackage[all]{xy}
\usepackage[T1]{fontenc}
\xyoption{line}
\usepackage{euscript}
\newtheorem{theorem}{Theorem}[section]
\newtheorem{cor}[theorem]{Corollary}
\newtheorem{lem}[theorem]{Lemma}
\newtheorem{prop}[theorem]{Proposition}

\theoremstyle{definition}

\newtheorem{example}[theorem]{Example}
\newtheorem{defi}[theorem]{Definition}
\newtheorem{rem}[theorem]{Remark}

\newcommand{\Alg}{\mathsf{dgAlg}}
\newcommand{\Mod}{\mathsf{Mod-\!}}
\newcommand{\D}{\mathsf{D}}
\newcommand\Perf{\mathsf{Perf}}

\numberwithin{equation}{section}

\DeclareMathOperator{\NRk}{NRk}
\DeclareMathOperator{\cone}{cone}
\DeclareMathOperator{\BQ}{\mathbb Q}
\DeclareMathOperator{\BZ}{\mathbb Z}
\DeclareMathOperator{\BR}{\mathbb R}
\DeclareMathOperator{\proj}{\mathsf{Proj}}
\DeclareMathOperator{\fp}{fp}
\DeclareMathOperator{\Hom}{Hom}
\DeclareMathOperator{\Map}{Map}
\DeclareMathOperator{\RHom}{RHom}
\DeclareMathOperator{\End}{End}
\DeclareMathOperator{\REnd}{REnd}

\DeclareMathOperator{\Ker}{Ker}

\DeclareMathOperator{\Loc}{Loc}

\DeclareMathOperator{\Spec}{Spec}

\newcommand{\noproof}{\begin{flushright} \ensuremath{\square}
\end{flushright}}
\def\ground{\mathbf{k}}
\def\C{\mathcal{C}}

\def\Z{\mathbb{Z}}
\def\R{\mathbb{R}}

\def\id{\operatorname{id}}

\begin{document}
\begin{abstract}We introduce the notion of a rank function on a triangulated category $\C$ which generalizes the Sylvester
rank function in the case when $\C=\Perf(A)$ is the perfect derived category of a ring $A$. We show that rank functions are closely related
to functors  into simple triangulated categories and classify
Verdier quotients into simple triangulated categories in terms of particular rank functions called \emph{localizing}. 
If $\C=\Perf(A)$ as above, localizing rank functions also classify 
finite homological epimorphisms from $A$ into differential graded skew-fields or, more generally, differential graded Artinian rings. To establish these results, we develop the theory of derived localization of differential graded algebras at thick subcategories of their perfect derived categories. This is a far-reaching generalization of Cohn's matrix localization of rings and has independent interest.
	\end{abstract}
\title[Derived rank functions]{Rank functions on triangulated categories}
\author{J. Chuang  \and A.~Lazarev}
\thanks{This work was partially supported by EPSRC grants EP/T029455/1 and  	
EP/T030771/1}
\address{Department of Mathematics\\City, University of London\\London
EC1V 0HB\\UK}
\email{J.Chuang@city.ac.uk}
\address{University of Lancaster\\ Department of
Mathematics and Statistics\\Lancaster, LA1 4YF, UK.}
\email{a.lazarev@lancaster.ac.uk}  
\maketitle
\tableofcontents
\section{Introduction}
The dimension of a vector space $V$ over a (possibly skew-)field $K$ is a basic characteristic of $V$ and it is an elementary fact that
it is an invariant, i.e. does not depend on the choice of a basis in $V$. However, further generalizations with $K$ replaced with a noncommutative ring $A$,
and $V$ replaced with an $A$-module, are not straightforward. Indeed, there are examples of rings such that their free modules do not have
a well-defined dimension, or rank. One possibility to obviate this difficulty is to start with a homomorphism $f:A\to K$ where $K$ is a skew-field (which allows one to associate to an $A$-module a $K$-module by tensoring up) and define the rank of an $A$-module as the rank of the corresponding $K$-module. Different homomorphisms $f$ give rise to possibly different ranks. This suggests that ranks are closely related to homomorphisms into skew-fields. This was made precise
by Cohn and Schofield \cite{Cohn95, Sch86} by showing that maps $f$ as above are in one-to-one correspondence with certain functions, called Sylvester rank functions,
on finitely presented $A$-modules, defined intrinsically. 

A crucial part of the Cohn-Schofield theory is the method of \emph{matrix localization}. Given a ring $A$ and a matrix $M$ with entries in $A$, there exists another ring $A[M^{-1}]$ supplied with a map $A\to A[M^{-1}]$ such that the matrix $M$ becomes invertible over $A[M^{-1}]$; moreover $A[M^{-1}]$ is universal in the sense that any other ring having this property factorizes uniquely through $A[M^{-1}]$. If $M$ is a $1\times 1$ matrix i.e. an element $s\in A$, then $A[M^{-1}]$ reduces to $A[s^{-1}]$, the usual localization of $A$ at $s$. Furthermore, if $A$ is a commutative ring, then general matrix localization reduces to inverting a single element, namely the determinant of $M$; however in the general noncommutative case no such reduction is possible.

Our original motivation was to rework the Cohn-Schofield theory in a way intrinsic to the derived category of $A$. To this end, we formulate
the notion of a rank function on an arbitrary triangulated category $\C$. Compared to the Sylvester rank function, our definition is simpler and, arguably, more natural. In the case $\C=\Perf(A)$ for a ring $A$, it subsumes that of the Sylvester rank function but \emph{does not reduce to it}. The `exotic' rank functions on $\Perf(A)$ (i.e. those that are not equivalent to Sylvester rank functions) are related to maps from $A$ into \emph{graded} skew-fields or \emph{graded} simple Artinian rings in the \emph{homotopy category of differential graded (dg) rings}.  Recall that  maps in the homotopy category of dg rings are computed by replacing the source with a cofibrant dg ring. This needs to be done even if the source is an ordinary ring. In other words, such maps are invisible on the classical level.

Our notion of a rank function on $\C$ appears close, albeit certainly not equivalent, to the notion of \emph{cohomological functions} on $\C$ in the sense of Krause \cite{Krause16}. The precise relationship between the two notions is unclear at the moment and we hope to return to this question in future.

Apart from the purely algebraic motivation described above, Sylvester rank functions are of great relevance to geometric group theory and, in particular, to various versions of the Atiyah conjecture, cf. \cite{Jaikin19} for a survey of recent results in this direction. A natural question, that we also leave for future investigation, is whether our notion of a rank functions can be usefully exploited in that context.

One unexpected source of rank functions on triangulated categories turns out to be Bridgeland's `stability conditions' \cite{Bridgeland07}. Namely, it turns out that there is a continuous map from the space of stability conditions to that of rank functions
on the same triangulated category. Here we limit ourselves with merely recording this observation but it undoubtedly deserves further study.

Next, we need to develop an analogue of matrix localization in a homotopy invariant way. Even when one is interested in inverting 
only one element in a ring, i.e. a $1\times 1$ matrix, care is needed since this operation is not exact in the noncommutative context.
The notion of derived localization (in this restricted sense) was developed in the previous work of the authors \cite{BCL}. In the present paper we build on this work to construct derived localization of a dg ring $A$ with respect to an arbitrary thick subcategory of $\Perf(A)$. This extends the notion of matrix localization since the latter is the nonderived version of the localization with respect to the thick subcategory generated by a collection of free complexes of length 2.  In contrast with inverting  a single element, even for commutative rings general derived localization may have nontrivial derived terms. However, for hereditary algebras, derived localization reduces to the non-derived notion.

A rank function $\rho$ on a triangulated category $\C$ has a kernel $\Ker(\rho)$, the thick subcategory of $\C$ consisting of objects of rank zero. We describe those rank  functions for which the Verdier quotient $\C/\Ker(\rho)$ is simple i.e. equivalent to the perfect derived category of a graded skew-field. These are the so-called \emph{localizing} rank functions. It turns out that localizing rank functions classify homotopy classes of homological epimorphisms from dg rings into dg fields and dg simple Artinian rings. We call such homological epimorphisms \emph{derived fraction fields}. Strikingly, ordinary rings
(even finite dimensional algebras over fields) have nontrivial derived fraction fields.

There are, of course, exact functors from a given triangulated category $\C$ to a simple one that are not Verdier quotients; these induce (non-localizing) rank functions on $\C$. It is an open question to which extent such functors
 are captured by rank functions.  

The paper is organized as follows. In Section \ref{section:rankdef} the notion of a rank function on a triangulated category is introduced, in two equivalent ways: as a function on objects or on morphisms satisfying appropriate conditions, as well as its refinement, a $d$-periodic rank function where $d=1,2,\ldots,\infty$ (an ordinary rank function is then 1-periodic). We show how rank functions can be constructed using functors into simple triangulated categories and (very briefly) stability conditions.

In Section \ref{section:Sylvester} we prove that $\infty$-periodic rank functions on the perfect derived category of ordinary rings
are nothing other than Sylvester rank function whereas Section \ref{section:further} establishes various further properties of rank functions reminiscent of those for Sylvester rank functions.

Section \ref{section:localization} develops the theory of derived localization for dg algebras at thick subcategories of their perfect derived categories and compares it with the non-derived notion. This section has an independent interest and can be read independently of the rest of the paper. Finally Section \ref{section:localizing} introduces the notion of a localizing rank function and shows that they describe Verdier quotients into simple triangulated categories and
homotopy classes of homological epimorphisms from dg algebras into dg fields and dg simple Artinian rings.
\subsection{Acknowledgement} We are grateful to Marc Stefan and the anonymous referee for pointing out various inaccuracies and a host of useful comments.
\subsection{Notation and conventions}Throughout this paper we work with homologically graded chain complexes over a fixed unital commutative ring $\ground$. Unadorned tensor products and Homs will be assumed to be taken over $\ground$. For a chain complex  $A$ 
 we denote by $\Sigma A$ its suspension given by 
$(\Sigma A)_i = A_{i-1}$. The signs $\simeq$ and $\cong$ will stand for quasi-isomorphisms and isomorphisms of chain complexes respectively.

We will normally use the abbreviation `dg' for `differential graded'. By `dg algebra' we will mean `dg associative unital algebra' over $\ground$. The category of dg algebras $\Alg$ has the structure of a closed model category (with weak equivalences being multiplicative quasi-isomorphisms) and so it makes sense to talk about homotopy classes of maps between dg algebras. A given dg algebra $B$ together with a dg algebra map $A\to B$ (not necessarily central) will be referred to as an $A$-algebra. The category $A\downarrow\Alg$ of $A$-algebras is likewise a closed model category. The category of right dg $A$-modules over a dg algebra $A$ will be denoted by $\Mod A$ and we will refer to its objects as $A$-modules. The category $\Mod A$ is a closed model category whose homotopy category is denoted by $\D(A)$, the \emph{derived category} of $A$. We chose to focus (for notational convenience) on right modules but will occasionally use left modules as well, making sure that no confusion would arise. 

For $A$-modules $M,N$ we denote by $\Hom_A(M,N)$ the chain complex of $A$-linear homomorphisms $M\to N$.  We write $\RHom_A(M,N)$ for the corresponding derived functor obtained by
replacing $M$ with a cofibrant $A$-module quasi-isomorphic to $M$. If $M=N$ we will write $\End_A(M)$ and $\REnd_A(M)$ for $\Hom_A(M,N)$ and $\RHom_A(M,N)$ respectively.

For a (right) $A$-module $M$ and a left $A$-module $N$ we write $M\otimes_AN$ and $M\otimes_A^L N$ for their tensor product and derived tensor product respectively. For two dg algebras $B$ and $C$ supplied with a dg algebra maps $A\to B$ and $A\to C$ their \emph{free product} or \emph{pushout} will be denoted by $B*_AC$, and its derived version -- by $B*^L_AC$.

A description of the closed model categories of algebras and modules convenient for our purposes is given in \cite{BCL}, except for the minor difference that \emph{left} modules are treated in that paper whereas \emph{right} modules are emphasized here. 

We will freely use the language of triangulated categories and their localizations, cf. \cite{Krause10loc} for an overview. If $\C$ is a triangulated category with translation functor $\Sigma$ and $\Sigma^d\cong\id$ we say that $\C$ has period $d$; if no such $d$ exists $\C$ is said to have infinite period. Given a triangulated category $\C$, its full triangulated subcategory $S$ is \emph{thick} if it is closed with respect to taking retracts. In this situation one can form the \emph{Verdier quotient} $\C/S$, a triangulated category supplied with a (triangulated) functor $j:\C\to \C/S$ whose kernel is $S$ and universal with respect to this property. A triangulated subcategory $S$ of $\C$ is called \emph{localizing} if it contains all small coproducts; in that case the Verdier quotient often admits a right adjoint $i:\C/S\to\C$, and then the endofunctor $L:=i\circ j:\C\to C$ is called the (Bousfield) localization functor with respect to $S$. It is necessarily idempotent: $L^2\cong L$ and for any $X\in \C$ the natural map $X\to L(X)$ is called the \emph{localization} of $X$ (with respect to $S$).  For a collection of objects $S$ of $\C$ we will denote by $\Loc(S)$ the smallest localizing subcategory of $\C$ containing $S$; we say that $S$ generates $\Loc(S)$.

We say that an object $X$ is a (classical) generator of the triangulated category $\C$ if $\C$ is the smallest thick subcategory of $\C$ containing $X$.

For a triangulated category $C$ a \emph{perfect} (or compact) object $X$ is characterized by the property 
\[\Hom_{\C}(X,\oplus_{i\in I} X_i)\cong \oplus_{i\in I}\Hom_{\C}(X,X_i)\] for any collection of $X_i,i\in I$ of $A$-modules indexed by a set $I$. The full subcategory of $\D(A)$ consisting of perfect $A$-modules will be denoted by $\Perf(A)$; note that $A$ is a generator of $\Perf(A)$. 

A \emph{dg category} is understood to be a category enriched over dg $\ground$-modules. Thus, for two objects $X_1$ and $X_2$ in a dg category $C$ we have a dg space of homomorphisms ${\Hom}(X_1,X_2)$ and composition is a dg map. The homotopy category $H_0(C)$ of the dg category $C$ has the same objects as $C$ and for two objects $X_1, X_2$ in $C$ we have $\Hom_{H_0(C)}(X_1,X_2):=H_0[{\Hom}_C(X_1,X_2)]$.

A dg functor $F:C\to C^\prime$ between two dg categories is \emph{quasi-essentially surjective} if $H_0(F):H_0(C)\to H_0(C^\prime)$ is essentially surjective and \emph{quasi-fully faithful} if $F$ induces quasi-isomorphisms on the $\Hom$-spaces; if both conditions are satisfied then $F$ is called a \emph{quasi-equivalence}. 

\section{Rank functions on triangulated categories}\label{section:rankdef}
Let $\C$ be a triangulated category. 
\begin{defi}\label{def:objectrank1per} A {\em rank function} on $\C$ is an assignment to each object $X$ of $\C$ of a nonnegative real number $\rho(X)$, such that the following conditions hold.
	\begin{enumerate}
		\item[] {\bf Translation invariance:} for any object $X$, we have
		\begin{equation}\tag{O1}\label{O1}\rho(\Sigma X)= \rho(X);\end{equation}
		\item[] {\bf Additivity:} for any objects $X$ and $Y$, we have
		\begin{equation}\tag{O2}\label{O2}\rho(X\oplus Y)=\rho(X)+\rho(Y);\end{equation}
		\item[] {\bf Triangle inequality:} for any exact triangle $X \to Y \to Z \rightsquigarrow$,
		we have 
		\begin{equation}\tag{O3}\label{O3}\rho(Y)\leq \rho(X) +\rho(Z)\end{equation}
		\end{enumerate}	
\end{defi}

\begin{rem} Conditions (\ref{O1}) and (\ref{O3}) may be replaced by the statement that for any exact triangle $X\to Y \to Z \rightsquigarrow$, the triple  $(\rho(X), \rho(Y), \rho(Z))$ is triangular, i.e. it is is composed of the side lengths of a (possibly degenerate) planar triangle. Indeed, one only needs to show that the latter condition together with  (\ref{O2}) implies (\ref{O1}).  Consider the exact triangle  $X \xrightarrow{0} X \to X\oplus \Sigma X \rightsquigarrow$. Then the triangularity condition together with (\ref{O2}) implies that $$\rho(\Sigma X\oplus X)=\rho(\Sigma X)+\rho(X)\leq 2\rho(X)$$ so that $\rho (\Sigma X)\leq \rho(X)$. Similarly the  exact triangle $\Sigma X\oplus X\to \Sigma X\xrightarrow{0}\Sigma X\rightsquigarrow$ implies that  $\rho (X)\leq \rho(\Sigma X)$ and it follows that $\rho(X)=\rho(\Sigma X)$.
\end{rem}	 

Rank functions on triangulated categories may alternatively be defined as functions on morphisms, as follows.

\begin{defi}\label{def:morphismrank1per}
	A  {rank function} on $\C$ is an assignment to each morphism $f$ in $\C$ of a  nonnegative real number $\rho(f)$, such that the following conditions hold.	
	\begin{itemize}
		\item[]{\bf Translation invariance:} for any morphism $f$, we have
		\begin{equation}\tag{M1}\label{M1}\rho(\Sigma f)= \rho(f);\end{equation}
		\item[] {\bf Additivity:} for any morphisms $f$ and $g$, we have
		\begin{equation}\tag{M2}\label{M2}\rho(f\oplus g)=\rho(f)+\rho(g);\end{equation}
		\item[] {\bf Rank-nullity condition:} for any exact triangle $X \xrightarrow{f} Y \xrightarrow{g} Z \rightsquigarrow,$
		we have 
		\begin{equation}\tag{M3}\label{M3}\rho(f)+\rho(g)=\rho(\id_Y);\end{equation}
	\end{itemize} 
\end{defi}	
The translation between the two definitions is given by the following formulae:
\begin{equation}\phantomsection\label{eq:morphismtoobject}
\rho(X)=\rho(\id_X)
\end{equation}
\begin{equation}\label{eq:objecttomorphism}
\rho(f:X\to Y) = \frac{\rho(X)+\rho(Y) - \rho(\cone(f))}{2}.
\end{equation}

\begin{prop}\label{prop:equivalence} Definitions \ref{def:morphismrank1per} and \ref{def:objectrank1per} are equivalent.
	\end{prop}
	\begin{proof} Given a nonnegative function $\rho$ on morphisms in $\C$ satisfying the three conditions of Definition~\ref{def:morphismrank1per}, the rule (\ref{eq:morphismtoobject}) defines a function on objects which is clearly nonnegative, and translation invariance (\ref{O1}) and additivity (\ref{O2}) follow immediately from the corresponding properties (\ref{M1}) and (\ref{M2}). Finally, take an exact triangle $X\xrightarrow{f} Y \xrightarrow{g} Z \rightsquigarrow$ in $\C$, with rotated exact triangles
$Y\xrightarrow{g} Z \xrightarrow{h} \Sigma X \rightsquigarrow$ and
$\Sigma^{-1}Z \xrightarrow{\Sigma^{-1}h} X \xrightarrow{f} Y \rightsquigarrow$. We have:
\begin{eqnarray*}
\rho(X)+\rho(Z)-\rho(Y) & = & \rho(\id_X)+\rho(\id_Z)-\rho(\id_Y) \\
& = & \left(\rho(\Sigma^{-1}h)+\rho(f)\right)
+ \left(\rho(g)+\rho(h)\right)
-\left(\rho(f)+\rho(g)\right) \\
& = & 2\rho(h) \\
& \geq & 0, \\ 
\end{eqnarray*}
by (\ref{M1}) and (\ref{M3}), which establishes (\ref{O3}).

Conversely, given a nonnegative function $\rho$ on objects of $\C$ satisfying the three conditions of Definition~\ref{def:objectrank1per}, define a function on morphisms by formula (\ref{eq:objecttomorphism}). For $f:X\to Y$, we have an exact triangle $Y\to\cone(f) \to \Sigma X\rightsquigarrow$, so (\ref{O1}) and (\ref{O3}) imply that $\rho(f)\geq 0$. Properties (\ref{M1}) and (\ref{M2}) are consequences of (\ref{O1}) and (\ref{O2}), respectively, and, finally, given an exact triangle
	 $X\xrightarrow{f} Y \xrightarrow{g} Z \rightsquigarrow$,
	 we confirm property \ref{M3}:
	 \begin{eqnarray*}
\rho(f)+\rho(g) & = &
\frac{\rho(X)+\rho(Y) - \rho(Z)}{2}
+
\frac{\rho(Y)+\rho(Z) - \rho(\Sigma X)}{2} \\
& = & \rho(Y) \\
& = & \rho(\id_Y), \\ 
\end{eqnarray*}
using \ref{O1}.
\end{proof}	
In view of Proposition \ref{prop:equivalence}, we regard a rank function on $\C$ as a function on both objects and morphisms, related by equations (\ref{eq:morphismtoobject}) and (\ref{eq:objecttomorphism}).

	We call a  rank function $\rho$ 
	\begin{itemize}
		\item {\em object-faithful}, if for all nonzero objects $X$ we have $\rho(\id_X)\neq 0$. 
		\item {\em morphism-faithful}, if for all nonzero maps $f$ we have $\rho(f)\neq 0$; 
		\item {\em integral}, if $\rho(f)\in\Z$ for all $f$.
		\item {\em prime} if $\rho$ is integral and $\C$ admits a generator $X$ such that $\rho(\id_X)=1$.			
	\end{itemize}

\begin{rem}\label{rem:clarification}
	 In the case $\C$ is a \emph{tensor} triangulated category (i.e it has a symmetric monoidal structure compatible with its triangulation, cf. \cite[Appendix A]{Hov97}), it makes sense to require additionally that a rank function $\rho$ is \emph{multiplicative} in the sense that $\rho(\id_X \otimes \id_Y) = \rho(\id_X)\rho(\id_Y)$ for any two objects $X,Y\in \C$. Multiplicative rank functions are likely to be of relevance to \emph{tensor triangular geometry}, \cite{Balmer05} but will not be considered in this paper.

	\end{rem}
\subsection{Periodic rank functions}\label{periodic} We will consider a certain refinement of the notion of a rank function defined above.

Let $R$ be an ordered commutative ring and fix $d\in\{1,2,\ldots\}\cup\{\infty\}$.
Put
$$R(d):=
\begin{cases}
R[q,q^{-1}], & \text{if } d=\infty \\
	R[q]/(q^d-1), & \text{if } d<\infty \\ 
\end{cases}\quad \text{and}\quad
R_{\geq 0}(d):=
\begin{cases}
	R_{\geq 0}[q,q^{-1}], & \text{if } d=\infty \\
	R_{\geq 0}[q]/(q^d-1), & \text{if } d<\infty \\ 
\end{cases}
$$
For $\phi,\psi\in R(d)$, we write $\phi\geq \psi$ to mean $\phi-\psi \in R_{\geq 0}(d)$. Given two integers $q, q^{\prime}$ where $d$ is divisible by $d^{\prime}$ or $d=\infty$, there is an obvious reduction map $\pi_{d, d^\prime}:R(d)\to R(d^{\prime})$. We will mostly be interested in the case $R=\R$ or $R=\Z$.

\begin{defi}\label{def:rankperiodic}
	A  {\em $d$-periodic rank function} on a triangulated category $\C$ is an assignment to each morphism $f$ in $\C$ of 
	$\rho(f)\in\BR_{\geq 0}(d)$, such that axioms (\ref{M2}) and (\ref{M3}) hold, axiom (\ref{M1}) gets modified as follows:
		
{\bf Translation invariance:} for any morphism $f$, we have
	\begin{itemize}\item[]	\begin{equation}\tag{Mp1}\label{Mp1}\rho(\Sigma f)= q\rho(f);\end{equation}\end{itemize}
and, in addition, the following axioms hold:
\begin{itemize}		
		\item[] {\bf Triangular inequality:} for all morphisms $f$, $g$ and $h$ in $\C$ such that $g$ and $h$ share the same domain and $f$ and $h$ share the same codomain, we have \begin{equation}\tag{M4}\label{M4}\rho\left(\begin{matrix} f & h \\ 0 & g\end{matrix}\right)\geq\rho(f)+\rho(g);\end{equation}
		\item[] {\bf Ideal condition:} for any morphisms $f$ and $g$ for which the composition $gf$ is defined, we have
		\begin{equation}\tag{M5}\label{M5}\rho(gf)\leq\rho(f) \quad \text{and} \quad \rho(gf)\leq\rho(g).\end{equation}
	\end{itemize} 
 \end{defi}
\begin{rem} Note that given two integers $d$ and $d^{\prime}$ where $d^{\prime}$ divides $d$ or $d=\infty$, a $d$-periodic rank function on $\C$ determines a $d^{\prime}$-periodic rank function on $\C$ via  the reduction map $\BR(d)\to \BR(d^\prime)$. In particular, any $d$-periodic rank function gives rise to a $1$-periodic rank function. Under this reduction, axiom (\ref{Mp1}) becomes (\ref{M1}) and we will see that often (e.g. when $d=1$) the axioms (\ref{M4}) and (\ref{M5}) are consequence of axioms (\ref{M1}), (\ref{M2}) and (\ref{M3}). In particular, the terms `rank function' and `1-periodic rank function' are synonymous. Furthermore,  often
a periodic rank function can be defined as a function on objects (e.g. we saw that it holds for $d=1$).
\end{rem}
A $d$-periodic rank function taking values in $\BZ_{\geq 0}(d)$ will be called \emph{integral}. The notions of a prime, object-faithful and morphism-faithful rank function obviously make sense in the $d$-periodic case.
\begin{prop}\label{prop:objectrank} Let $\rho$ be a $d$-periodic rank function on $\C$. For all objects $X$ in $\C$, put
	$$\rho(X):=\rho(\id_X) \in \BR_{\geq 0}(d).$$	
	Then $\rho$ satisfies the additivity axiom (\ref{O2}) of the rank function whereas axioms (\ref{O1}), (\ref{O3}) get refined as follows:
	\begin{enumerate}
		\item[] {\bf Translation invariance:} for any object $X$ in $\C$, we have
		\begin{equation}\tag{Op1}\label{Op1}\rho(\Sigma X)= q\rho(X);\end{equation}
		\item[] {\bf Triangle inequality:} for any exact triangle $X \to Y \to Z \rightsquigarrow$ in $\C$
		we have 
		\begin{equation}\tag{Op3}\label{Op3}\rho(X)-\rho(Y)+\rho(Z)=(q+1)\phi,\end{equation}
		for some $\phi\in\BR_{\geq 0}(d)$. More precisely, we can take 
		$\phi=\rho(f)$, where $f:\Sigma^{-1}Z\xrightarrow{f} X$ is the connecting morphism for the triangle.	\end{enumerate}	
	Moreover, if $\rho$ is integral, i.e. $\rho(f)\in\BZ_{\geq 0}(d)$ for all morphisms $f$ in $\C$, then  $\rho(X)\in\BZ_{\geq 0}(d)$  for any $X\in\C$.
\end{prop}
\begin{proof} This is a straightforward modification of the first part of proof of  Proposition \ref{prop:equivalence}.
In particular 	conditions (\ref{Op1}) and (\ref{O2}) follow immediately from (\ref{Mp1}) and (\ref{M2}). Furthermore, given an exact triangle $\Sigma^{-1}Z
	\xrightarrow{f} X \xrightarrow{g} Y \xrightarrow{h} Z$, we have, using (\ref{M3}) and (\ref{Mp1}),
	\begin{align*}
	\rho(X)-\rho(Y)+\rho(Z) & = (\rho(f) + \rho(g))-(\rho(g)+\rho(h))+(\rho(h)+\rho(\Sigma f)) \\
	& = (q+1)\rho(f),
	\end{align*}
	as desired. The claim about integrality is likewise clear.
\end{proof}
\begin{rem}
Note that in order to obtain axioms (\ref{Op1}), (\ref{O2}) and (\ref{Op3}), we only make use of conditions (\ref{Mp1}), (\ref{M2}) and (\ref{M3}) of Definition \ref{def:rankperiodic}.
\end{rem}
The following useful observation relates the rank of morphism with that of its source and target.
\begin{lem}\label{lem:boundmorphismrank}
Let $\rho$ be a $d$-period rank function on $\C$. For any morphism $f:X\to Y$, we have $\rho(f) \leq \rho(X)$ and $\rho(f) \leq \rho(Y)$.
\end{lem}
\begin{proof}
	Consider the factorizations $f=f\circ\id_X$ and $f=\id_Y\circ f$ in the ideal condition (\ref{M5}).
	\end{proof}

To check condition \ref{Op3} when $d=\infty$, the following obvious lemma is useful.
\begin{lem}\label{lem:obvious}
	A real Laurent polynomial $f(q)=\sum_{n\in\BZ} a_nq^n$ can be written as $f(q)=(1+q)\phi$ where $\phi$ has nonnegative coefficients if and only if
 for any $n\in\BZ$ it holds that $\sum_{i=0}^{\infty}(-1)^{i}a_{n+i}\geq 0$.
	\noproof
\end{lem}
It is natural to ask whether a rank function is determined by its values on objects (so as to obtain a periodic analogue Proposition \ref{prop:equivalence}). The following result gives a partial answer to that.
\begin{prop}\label{prop:objectmorphism}
Assume that $d=\infty$ or $d$ is odd and there is given an assignment to any object $X$ of $\C$ of an element $\rho(X)\in\R(d)$ satisfying conditions (\ref{Op1}), (\ref{O2}) and (\ref{Op3}). Then, for any $f:X\to Y$ in $\C$ the formula
\begin{equation}\label{eq:objectmor} \rho(f):=\frac{\rho(Y)-\rho(\cone(f))+q\rho(X)}{q+1}\end{equation}	
determines a $d$-periodic rank function on $\C$.  If, in addition, $\rho(X)\in\BZ(d)$ and the polynomial $\phi$ figuring in triangle inequality (\ref{Op3}), belongs to $\BZ(d)$, then the obtained rank function is integral.
\end{prop}
\begin{proof}
We will only treat the case of real rank functions; the integral case is obtained completely analogously. Note that $q+1$ is invertible in $\BR(d)$ if $d$ is odd, and a
non-zero-divisor (like all non-zero elements) if $d=\infty$. By Proposition \ref{prop:objectrank}, the element $\rho(Y)-\rho(\cone(f))+q\rho(X)$ is divisible by $q+1$ and it follows that it is \emph{uniquely} divisible. 
	Thus, $\rho(f)\in \BR(d)$ for any $X\in \C$. The conditions (\ref{Mp1}), (\ref{M2}) and (\ref{M3}) of Definition \ref{def:rankperiodic} are proved by the same argument as the second part of Proposition \ref{prop:equivalence}. 
	
To check (\ref{M5}), suppose we are given 
morphisms $f:X\to Y$ and $g:Y\to Z$.
We have 
\begin{align*}
\rho(f)-\rho(gf)&=\frac{q\rho(X)+\rho(Y)-\rho(\cone(f))}{q+1}-\frac{q\rho(X)+\rho(Z)-\rho(\cone(gf  ))}{q+1}\\
&=\frac{\rho(Y)-\rho(\cone(f)\oplus Z)+\rho(\cone(gf)))}{q+1}\\
&\geq 0,
\end{align*}
where the inequality is deduced from Lemma~\ref{lemma:weakNeeman} below and (\ref{Op3}).

Finally, we check (\ref{M4}). Suppose we have morphisms $f\colon X\to Y$, $g\colon Z \to W$ and $h:Z\to Y$
in $\C$.  There is an exact triangle in $\C$ of the form
$$\cone(g)\to\cone\left(\begin{matrix} f & h \\ 0 & g\end{matrix}\right)\to\cone(f)\rightsquigarrow.$$
Hence, by (\ref{O2}) and (\ref{Op3}),
\begin{align*}
\rho\left(\begin{matrix} f & h \\ 0 & g\end{matrix}\right) & =
\frac{\rho(Y\oplus W)-\rho\left(\cone\left(\begin{matrix} f & h \\ 0 & g\end{matrix}\right)\right)
	+q\rho(X\oplus Z)}{q+1} \\
& \geq \frac{\rho(Y)+\rho(W)-\rho(\cone(f))-\rho(\cone(g))+q\rho(X)+q\rho(Z)}{q+1} \\
& = \rho(f) + \rho(g).
\end{align*}	 
\end{proof}

\begin{lem}\label{lemma:weakNeeman}
	For any pair of composable morphisms $f:X\to Y$ and $g:Y\to Z$ in a triangulated category, there exists an exact triangle of the form
	$$Y\to \cone(f)\oplus Z \to \cone(gf)\rightsquigarrow.$$
\end{lem}

\begin{proof}There exists the following commutative diagram where the right square is a homotopy pushout and the rows are exact triangles:
\[
\xymatrix{
X\ar@{-->}^{\cong}[d]\ar^f[r]&Y\ar^g[d]\ar[r]&\cone(f)\ar[d]\ar@{~>}[r]&\\
X'\ar[r]&Z\ar[r]&W\ar@{~>}[r]&
}
\] 
By a version of the octahedral axiom the dotted arrow exists making the whole diagram commutative and the bottom row into an exact triangle. It follows that $W\cong\cone(gf)$ and the desired result follows.	
%
	\end{proof}

\subsection{Stability conditions and rank functions}
Recall that a \emph{stability condition} on a triangulated category $\C$, cf. \cite{Bridgeland07} is a pair $(\mathcal{P},Z)$ where
$\mathcal{P}=\mathcal{P}(\phi),\phi\in\BR$ is a slicing on $\C$, a collection of subcategories of $\C$ with properties modelled on the Postnikov truncations in the category of chain complexes and a \emph{central charge}, that is a homomorphism $Z:K_0(\C)\to\mathbb{C}$ compatible with the slicing in a suitable way. Every nonzero object $E$ of $\mathcal C$ has a filtration $0=E_0\to\ldots\to E_n=E$ so that 
$A_i:=\cone(E_{i-1}\to E_{i})$  belongs to $P(\phi_i)$ and $\phi_1>\ldots>\phi_n$; one sets $\phi^-_{\mathcal P}(E)=\phi_1$ and $\phi^+_{\mathcal P}(E)=\phi_n$.  
The \emph{mass} of $E$ is defined as $m_\sigma(E)=\sum_{i=1}^n|Z(A_i)|$.

More generally, one can introduce a parameter and define
$$m_{\sigma, t}(E) = \sum |Z(A_i)| e^{\phi_i t}$$
 cf.\cite[Section 4.5]{Dimitrov14}.

\begin{prop}
Given a stability condition $\sigma=(\mathcal{P},Z)$ on a triangulated category $\C$,
the mass $m_{\sigma}$ defines an object faithful  rank function on $\C$.
\end{prop}
  
 \begin{proof}
Translation invariance (\ref{O1}), additivity (\ref{O2}) and object-faithfulness are obvious.
The triangle inequality for $ m_{\sigma, t}(E)$ is proved in 
\cite[Proposition 3.3]{Ikeda16} for all $t$ and (\ref{O3}) follows by taking $t=0$.
\end{proof}
\begin{rem}
The last result suggests that a rank function may serve as a replacement for a stability condition on $\C$ which exists even in the absence of a $t$-structure (e.g. when $\C$ is periodic).
\end{rem}
The set $\operatorname{Stab}(\C)$ of stability conditions on $\C$ is a topological space. The topology may be induced by the
generalized metric (allowed to assume infinite values):
$$d(\sigma_1,\sigma_2) = \sup_{0\neq E\in\C}\left\{|\phi^-_{\sigma_2}(E)-\phi^-_{\sigma_1}(E)|, |\phi^+_{\sigma_2}(E)-\phi^+_{\sigma_1}(E)|, 
\left|\log\frac{m_{\sigma_2}(E)}{m_{\sigma_1}(E))}\right|\right\},$$
(cf. \cite{Bridgeland07}, p. 341).
In a similar (albeit much more obvious) way, 
the set $\operatorname{FRank}(\C)$ of object faithful rank functions on $\C$ is topologized by the generalized metric
$$ d(\rho_1,\rho_2)=\sup_{0\neq E\in\mathcal C}\left|\log\frac{\rho_1(E)}{\rho_2(E)}\right|.$$

On the set $\operatorname{Rank}(\C)$ of \emph{all} rank functions on $\C$, it is more useful, following Schofield \cite[Chapter 7]{Sch86}, to consider
the topology of pointwise convergence of real-valued functions on objects of $\C$.

The following result is immediate:

\begin{prop}
The maps
\begin{align*}
\operatorname{Stab}(\C)&\rightarrow\operatorname{FRank}(\C)\hookrightarrow  \operatorname{Rank}(\C) \\
 \sigma &\mapsto m_\sigma\
 \end{align*}
are continuous.\qed
	\end{prop}

In \cite{Bapat20}, it is suggested that a compactification of the quotient $\operatorname{Stab}(\C)/\mathbb{C}$ of $\operatorname{Stab}(\C)$ by a natural action of $\mathbb{C}$ may be constructed as the closure of its image in $\operatorname{Rank}(\C)/\BR^{\times}$. 
\subsection{Simple triangulated categories} The most basic examples of rank functions come from \emph{simple} triangulated categories.
\begin{defi}
	We say that a triangulated category $\C$ is \emph{simple} if every exact triangle in $\C$ is split (by which we mean that it is isomorphic to a direct sum of triangles having an isomorphism as one of their arrows) and $\C$ is generated as a triangulated category by one non-zero indecomposable object. 
\end{defi} 
Simple triangulated categories are in 1-1 correspondence with \emph{graded skew-fields}, i.e. graded rings whose non-zero homogeneous elements are invertible.
\begin{prop}\label{prop:simplefield} A triangulated category $\C$ is simple if and only if
	$\C$ is equivalent to $\Perf(K)$ where $K$ is a graded skew-field (i.e. the category of finite-dimensional vector $K$-spaces).
\end{prop}
\begin{proof}
	Let $\C$ be a simple triangulated category with an indecomposable generator $X$ and let $f:X\to \Sigma^n X$ be a morphism. Then the exact triangle
	$X\stackrel{f}{\to}\Sigma^nX\to \text{cone}({f})\rightsquigarrow$ is contractible, that is one of its arrows contains an isomorphism as a direct summand. Since $X$ is indecomposable, this can only happen if $f$ is an isomorphism or zero. We conclude that the graded ring $\End^*(X)$ is a graded skew-field and $\C$ is equivalent to the category of graded vector spaces over it. Conversely,
	all monomorphisms and epimorphisms of graded modules over a graded field $K$ split and it follows that 	$\Perf(K)$ is a simple triangulated category.
\end{proof}	
\begin{rem}
	Graded skew-fields can easily be classified in terms of ordinary skew-fields. Indeed, let $K$ be a graded skew-field; then $K_0$, its zeroth component is an (ungraded) skew-field. Suppose that $K_0\neq K$ and let $d$ be the smallest positive integer such that $K$ has a nonzero, hence invertible, element of degree $d$; denote this element by $t$.  Then $K$ is isomorphic as a $K_0$-module to $K_0[t, t^{-1}]$ and if $t$ is central then this isomorphism is multiplicative. In general $K$ will be isomorphic to a skew Laurent polynomial ring $K_0[t, t^{-1}]^\sigma$ where $\sigma$ is an automorphism of $K_0$ and the multiplication is determined by the rule $ta =\sigma(a)t$ for $a\in K_0$.
	
	The graded skew-field of the form $K=K_0[t, t^{-1}]^\sigma$ with $|t|=d$ will be called $d$-periodic and $d$ is called the \emph{period} of $K$; furthermore we adopt the convention that an ungraded skew-field $K=K_0$ has infinite period. The corresponding simple triangulated category $\Perf(K)$ is $d$-periodic i.e. the functor $\Sigma^d$ on $\Perf(K)$ is naturally isomorphic to the identity.
\end{rem}
Periodic rank functions on simple triangulated categories are essentially unique. More precisely, we have the following result.
\begin{prop}\label{prop:primesimple}
	The space of $d'$-periodic rank functions on a $d$-periodic simple triangulated category is isomorphic to
	$\BR_{\geq 0}(d')$ if $d=\infty$ or $d'$ is a divisor of $d$; otherwise it is $\{0\}$.
	If the space is nonzero, there exists a unique (up to multiplication by a power of $q$) prime rank function.
\end{prop}
\begin{proof}
	Let $\C$ be a simple $d$-periodic triangulated category with a generator $X$ and $\rho$ be a non-zero $d^\prime$-periodic rank function on $\C$; assume that $d^\prime<\infty$. Then $\rho(X)\in \BR_{\geq 0}(d')$ so that $d^\prime$ is the smallest integer with $q^{d^\prime}=1$. If $d<\infty$ then
	$X\cong\Sigma^dX$ and $\rho(X)=\rho(\Sigma^dX)=q^d\rho(X)$ so that $q^d=1$ from which it follows that $d^\prime$ divides $d$.  Furthermore, any other $d^\prime$-periodic rank function can be obtained from $\rho$ by multiplying it with an arbitrary element of $\BR_{\geq 0}(d')$ and rank functions corresponding to different elements in $\BR_{\geq 0}(d')$ will be different. The argument with $d^\prime=\infty$ is similar. Finally, if a non-zero rank function on $\C$ exists, then the condition $\rho(X)=1$ specifies it uniquely.
\end{proof}
\begin{cor}\label{cor:simplerank}
A simple triangulated category admits a unique prime rank function.
\end{cor}
\begin{proof}	
This follows from Proposition \ref{prop:primesimple} by specializing $q=1$.
\end{proof}
\begin{rem}
	Recall that by Proposition \ref{prop:simplefield} a simple triangulated category is equivalent to the category of finite-dimensional vector spaces over a (possibly graded) skew-field. Clearly the unique rank function of Corollary \ref{cor:simplerank} is just the dimension of a finite-dimensional vector space.
\end{rem}
\begin{prop}\label{prop:faithfulsimple}
If a triangulated category admits a morphism-faithful prime rank function then it is simple.
\end{prop}
\begin{proof}
	Let $\C$ be a triangulated category having a generator $X$ and a morphism-faithful rank function $\rho$ with $\rho(X)=1$. Let $f:\Sigma^n X\to X$ be any non-zero morphism of some degree $n$; by (\ref{M3}), $\rho(f)$ is either $1$ or $0$, by morphism-faithfulness the second possibility is ruled out so $\rho(f)=1$. Considering the exact triangle
	\[
	\xymatrix{\Sigma^{-1}Y\ar^h[r]&\Sigma^d X\ar^f[r]&X\ar^g[r]&Y}
	\]
	and using (\ref{M3}) again we conclude that $\rho(h)=\rho(g)=0$ and then by morphism-faithfulness $g=h=0$ and $f$ is an isomorphism. Thus, any nonzero element in $\End(X)$ is invertible, i.e. $\End(X)$ is a graded skew-field.
\end{proof}	
\begin{defi}A dg algebra $A$ is called a dg skew-field if its homology $H(A)$ is a graded skew-field. Similarly $A$ is a simple Artinian dg algebra if $H(A)$ is a graded simple Artinian algebra (i.e. it is isomorphic to the graded matrix algebra $M_n(K)$ over some graded skew-field $K$).
\end{defi}
The derived categories of dg skew-fields and dg Artinian simple algebras are all simple:
\begin{prop}\
\label{prop:simple}\begin{enumerate}\item  Let $A$ be a simple Artinian dg algebra so that $H(A)\cong M_n(K)$ where $K$ is a graded skew-field of period $d$ and $n$ is some integer. Then $\Perf(A)$ is a simple triangulated category of period $d$. \item Conversely, if $A$ is a dg algebra for which $\Perf(A)$ is simple, then $A$ is a simple dg Artinian algebra.
\end{enumerate}
\end{prop}
\begin{proof}
The homology graded ring $H(A)$ of $A$ is a graded matrix algebra over some graded skew-field $K$. A primitive idempotent $e\in H(A)$ determines a retract $X:=eA$ in $\Perf(A)$. The $H(A)$-module $H(X)$ is a simple generator of the category of $H(A)$-module. Given a perfect dg $A$-module $M$, its homology $H(M)$ is a graded $H(A)$-module that is a finite direct sum of $H(A)$-modules $H(X)$ and it is clear that this decomposition lifts to $\Perf(A)$ so that $M$ is a direct sum of copies of $X$ in $\Perf(A)$. Similarly, a map $M\to N$ of $A$-modules is determined by the map $H(M)\to H(N)$ of $H(A)$-modules and it follows that any exact triangle in $\Perf(A)$ splits. Thus, $\Perf(A)$ is simple and its period clearly coincides with that of $K$. This proves (1).

For (2) let $A$ be a dg algebra for which $\Perf(A)$ is simple; denote by $X$ an indecomposable generator of $\Perf(A)$ and assume without loss of generality that $X$ is a cofibrant $A$-module. Since every nonzero map $X\to\Sigma^n X$ is invertible in $\Perf(A)$, the dg algebra $B:=\End_A(X)$ is a dg skew-field. Then the functor $F:M\mapsto M\otimes_A X$ is a dg equivalence between the categories
of cofibrant left $A$-modules and cofibrant $B$-modules and therefore the dg algebras $A\cong\End_A(A)$ and $\End_B(F(A))=\End_B(X)$  are quasi-isomorphic. Since $B$ is a dg skew-field, the $B$-module $X$ is a direct sum of simple $B$-modules and thus, $\End_B(X)$ is a matrix algebra of a dg skew-field, i.e a dg simple Artinian ring.
\end{proof}
\begin{rem}
Another notion of a simple dg algebra was considered in a recent paper \cite{Orlov19}. Orlov's notion is much stronger than ours, in that it is quasi-isomorphic to an ordinary simple algebra. A simple Artinian dg algebra in our sense or even a dg skew-field of finite period is \emph{not} determined up to quasi-isomorphism by its homology algebra. Examples of non-formal $A_\infty$ algebras  whose homology algebras are skew-fields are not hard to construct (cf. for example even Moore algebras of \cite{Lazarev03}). This situation arises also in stable homotopy theory where Eilenberg-MacLane spectra of graded fields and  Morava $K$-theories have equivalent simple homotopy categories of modules yet are very different in many ways, e.g. they have inequivalent homotopy categories of bimodules. 
\end{rem}

\section{Rank functions on perfect derived categories of ordinary rings}\label{section:Sylvester}
We start with a short review of ordinary Sylvester rank functions, following \cite{Sch86}. This theory and its applications motivated us to develop its derived analogue. 
\subsection{Sylvester rank functions} 
Let $A$ be a ring. Denote by $\fp(A)$ the category of finitely presented $A$-modules, and by $\proj(A)$ the subcategory of finitely generated projective $A$-modules.

\begin{defi}\label{def:Sylvester-morphism}
	A {\em Sylvester morphism rank function} on $A$ associates to each morphism $f$ of $\proj(A)$ a {\em rank}
	$\rho(f)\in \BR_{\geq 0}$. It is required to satisfy the following conditions.
	\begin{itemize}
		\item[] {\bf Normalization:} \begin{equation}\tag{m1}\label{m1}\rho(\id_A)=1;\end{equation}
		\item[] {\bf Additivity:} for any morphisms $f$ and $g$, we have
		\begin{equation}\tag{m2}\label{m2}
			\rho(f\oplus g)=\rho(f)+\rho(g);
		\end{equation}
		\item[] {\bf Triangular inequality:} for all morphisms $f$, $g$ and $h$ such that $g$ and $h$ share the same domain and $f$ and $h$ share the same codomain, we have \begin{equation}\tag{m3}\label{m3}\rho\left(\begin{matrix} f & h \\ 0 & g\end{matrix}\right)\geq\rho(f)+\rho(g);\end{equation}
		\item[] {\bf Ideal condition:} for any morphisms $f$ and $g$ for which the composition $gf$ is defined, we have
		\begin{equation}\tag{m4}\label{m4}\rho(gf)\leq\rho(f) \quad \text{and} \quad \rho(gf)\leq\rho(g)\end{equation}
	\end{itemize}
\end{defi}

\begin{defi}\label{def:Sylvester-object}
	A {\em Sylvester module rank function} on $A$ associates to each finitely presented $A$-module $M$ a {\em rank} $\rho(M)\in\BR_{\geq 0}$, such that
	\begin{itemize}
		\item[] {\bf Normalization:} \begin{equation}\tag{o1}\label{o1}\rho(A)=1;\end{equation}
		\item[] {\bf Additivity:} for all finitely presented modules $M$ and $N$, we have
		\begin{equation}\tag{o2}\label{o2}\rho(M\oplus N) = \rho(M)+\rho(N);\end{equation}
		\item[] {\bf Triangle Inequality:} for any exact sequence $$L\to M \to N \to 0$$ of finitely presented $A$-modules, we have
		\begin{equation}\tag{o3}\label{o3}\rho(N) \leq \rho(M) \leq \rho(L)+\rho(N).\end{equation}
	\end{itemize}
\end{defi}

Let
$$P\xrightarrow{f} Q \to M \to 0,$$
be an exact sequence of $A$-modules, with $P,Q\in\proj(A)$, so that $M\in\fp(A)$. 
The formulas
$$\rho(f) = \rho(Q)-\rho(M) \quad \text{and} \quad \rho(M)=\rho(\id_Q)-\rho(f),$$
yield a 1-1 correspondence between Sylvester morphism rank functions on $A$ and Sylvester object rank functions on $A$. 
So we just call the resulting function defined on both morphisms and objects a {\em Sylvester rank function} on $A$.
\begin{example}
	Let $A$ be a skew-field; then the category $\fp(A)$ is just the category of finite-dimensional vector spaces over $A$. It is clear that $\fp(A)$ has a unique Sylvester rank function given by the dimension of a vector space; a triangulated analogue of this obvious result is Corollary \ref{cor:simplerank}.
\end{example}
\begin{rem}
	A Sylvester rank function on $A$ may be recovered from its values on homomorphisms $f:A^m\to A^n$ between free modules, since every finitely presented module is isomorphic to the cokernel of such a map. By additivity, one may even restrict to the case of square matrices, i.e.\ the case $m=n$.
\end{rem}

Let $S$ be a simple Artinian ring. Then $S$ is isomorphic to a matrix algebra $M_n(K)$ over a skew-field $K$. Every object of $\fp(S)=\proj(S)$ is a finite direct sum of copies of the simple module $V=K^n$, and the unique Sylvester rank function on $\fp(S)$ takes the value $\frac{1}{n}$ on $V$.

Given a homomorphism $A\to B$, we obtain a right exact functor $\fp(A)\to\fp(B): M \mapsto B\otimes_A M$. It is easy to confirm that any Sylvester rank function on $B$ pulls back via this functor to a Sylvester rank function on $A$. In particular any homomorphism $A\to S$ into a simple Artinian ring determines a canonical Sylvester rank function on $A$. Two homomorphisms $A\to S$ and $A \to S'$ into simple Artinian rings are considered equivalent if they become equal after composing with homomorphisms $S\to S''$ and $S'\to S''$ into a third simple Artinian ring. Equivalent homomorphisms clearly determine the same Sylvester rank function on $\fp(A)$.

\begin{theorem}\label{thm:main1}
	There is a 1-1 correspondence between equivalence classes of homomorphisms of $A$ into simple Artinian rings $S$ and $\BQ$-valued Sylvester rank functions on $A$ taking the value $0$ or $1$ on $A/rA$ for any integer $r$. In this correspondence $S$ is a skew-field if and only if the corresponding rank function is $\BZ$-valued.
\end{theorem}
\begin{proof}
	See \cite{Sch86}, Chapter 7, Theorems 7.12 and 7.14.
\end{proof}	
\begin{rem}\label{rem:technicalcondition} 
	Let $\rho$ be a Sylvester rank function on $A$. Since $0\leq \rho(A/rA)\leq 1$ for any integer $r$, the condition on $\rho(A/rA)$ is automatically satisfied if $\rho$ is $\BZ$-valued. The same is true for an arbitrary Sylvester rank function on $A$ when $A$ is an algebra over a field, for then  $A/rA$ is isomorphic to either $A$ or $0$. See \cite{Sch86}, p. 121 for a discussion on this issue.
\end{rem}	
We already explained how to obtain a rank function from a homomorphism into a simple Artinian ring $S$, and since $S$ is an algebra over a field, namely its centre, it will satisfy the additional condition, c.f. Remark \ref{rem:technicalcondition}. The reverse construction in the case of an integral Sylvester rank function $\rho$ is described as follows. The Cohn localization of $A_\rho$ of $A$ with respect to all morphisms $f:P\to Q$ in $\proj(A)$ such that $\rho(f)=\rho(P)=\rho(Q)$, can be shown to be local with residue skew-field, say, $K$, cf. \cite[Theorem 7.5]{Sch86}. The desired homomorphism is the composition $A\to A_\rho \to K$.

\subsection{Derived rank functions}
Let $A$ be an ordinary ring. We consider periodic rank functions on the triangulated category $\Perf(A)$.  A $d$-periodic rank function $\rho$ on $\Perf(A)$ is called {\em normalized} if $\rho(A)=1$. We will denote by $\NRk_d(\Perf(A))$ and $\operatorname{Syl}(A)$ the set of normalized $d$-periodic rank functions on $\Perf(A)$ and Sylvester rank functions on $\fp(A)$ respectively. In this section we give a comparison between $\NRk_d(\Perf(A))$ and $\operatorname{Syl}(A)$.

\begin{theorem}\label{thm:Sylvesterderived}
	Let $A$ be an ordinary ring and $d\in\{1,2,\ldots\}\cup\{\infty\}$. The restriction of any normalized $d$-periodic rank function on $\Perf(A)$ to $\proj(A)$
	is a Sylvester morphism rank function on $A$. In the case $d=\infty$, this determines a 1-1 correspondence between normalized $\infty$-periodic rank functions on $\Perf(A)$ and Sylvester rank functions on $A$.
\end{theorem}
\begin{proof}
	By additivity, any normalised periodic rank function takes values in $\BR$ on any object in $\proj(A)$, and therefore on any morphism in $\proj(A)$, by Lemma \ref{lem:boundmorphismrank}.  
	The first statement then follows from a straightforward unraveling of definitions. The stronger statement for $d=\infty$ is a consequence of the Proposition \ref{prop:extendrank} below.
\end{proof}	

\begin{lem} \label{lemma:perfrank2}Let $\rho$ be a normalized $d$-periodic rank function on $\Perf(A)$. 
	Let $X$ be a bounded complex of $A$-modules with finitely generated projective terms $X_n$. Then
	$$\rho(X) \leq \sum_{n\in\BZ} \rho(X_n)q^n.$$
\end{lem}
\begin{proof}
	This follows by induction on the length of $X$ from the triangle inequality (\ref{Op3}).
\end{proof}	
\begin{prop} \label{prop:extendrank} Let $A$ be an ordinary ring. Any normalized $\infty$-periodic rank function $\rho$ on $\Perf(A)$ is determined by its restriction to $\proj(A)$. More precisely, given a bounded complex $X=\{X_n\}$ of  finitely generated projective $A$-modules with a differential $d_n:X_n\to X_{n-1}$, we have
	\begin{equation} \phantomsection\label{eq:extendrank}
\rho(X) = \sum_{n\in \BZ} \left(\rho(X_n)-\rho(d_{n})-\rho(d_{n+1})\right)q^n.
\end{equation}
	Moreover, any Sylvester rank function on $A$ extends (uniquely) via this formula to 
	a rank function $\rho$ on $\Perf(A)$. This extension will be referred to as the \emph{derived} rank function on $\Perf(A)$. 
\end{prop} 
\begin{proof}
	Let us prove that given a Sylvester rank function $\rho$, formula (\ref{eq:extendrank}) gives a $\infty$-periodic rank function on $\Perf(A)$. Note that 
	 a bounded exact complex $X=\{\ldots\stackrel{d_n}{\leftarrow}X_n\stackrel{d_{n+1}}{\leftarrow}\ldots\}$ of finitely generated projective $A$-modules splits as a direct sum of elementary exact complexes of length two.
	 It is clear that for such a complex one has $\rho(X_n)=\rho(d_n)+\rho(d_{n+1})$ and it follows that $\rho(X)=0$.
	
	Given a quasi-isomorphism $f:X\to Y$ of bounded complexes of finitely generated projective $A$-modules we can factor it as
	$X\stackrel{i}\rightarrow Z\stackrel{p}{\rightarrow}Y$ where $Z$ is the  mapping cylinder of $f$. Since $X$ and $Y$ are chain deformation retracts of $Z$, it follows that $\rho(X)=\rho(Z)=\rho(Y)$. Therefore, $\rho$ is well-defined on $\Perf(A)$.
	
	It is clear that the function $\rho$ on $\Perf(A)$ given by (\ref{eq:extendrank}) satisfies the axioms (\ref{Op1}) and (\ref{O2}) of the rank function. To prove (\ref{Op3}), consider an exact triangle $X \to Y \to Z \rightsquigarrow$ in $\Perf(A)$. Without loss of generality, we can assume that the map $X\to Y$ is a cofibration of complexes of $A$-modules, in other words, it is a degree-wise split injection, and $Z\cong X/Y$ is a bounded complex of finitely generated projective $A$-modules. By the property (\ref{o2}) of a Sylvester rank function, we have for all $n\in \Z$:
	\begin{equation}\label{eq:shortexact}
		\rho(Y_n)=\rho(X_n)+\rho(Z_n).
	\end{equation}
	Denote by $d_n^X, d_n^Y,d_n^Z, n\in \Z$ the differentials in the complexes $X,Y$ and $Z$ respectively. Then by (m3) we have for all $n\in\Z$:
	\begin{equation}\label{eq:shortexact'}
		\rho(d_n^Y)\geq \rho(d_n^X)+\rho(d_n^Z).
	\end{equation}
	From  (\ref{eq:shortexact})  we deduce:
	\begin{equation}\label{eq:triangle}
		\rho(X)-\rho(Y)+\rho(Z)=\sum_{n\in\Z}(\rho(d_n^Y)+\rho(d_{n+1}^Y)-\rho(d_n^X)-\rho(d_{n+1}^X)-\rho(d_n^Z)-\rho(d_{n+1}^Z))q^n.
	\end{equation}
	Using (\ref{eq:shortexact'}) we conclude that the polynomial $\rho(X)-\rho(Y)+\rho(Z)$ satisfies the condition of Lemma \ref{lem:obvious}. It follows that $\rho(X)-\rho(Y)+\rho(Z)=(q+1)\phi$ with $\phi\in{\mathbb R}_{\geq 0}(d)$ as required. 
	
	Given a morphism $f:P\to Q$ in $\proj(A)$, (\ref{eq:extendrank}) gives
	$\rho(\cone(f)) = (\rho(P)-\rho(f))q + (\rho(Q)-\rho(f))$, consistent with (\ref{eq:objectmor}). It follows that the 
	$\infty$-periodic rank function on $\Perf(A)$ defined by (\ref{eq:extendrank})
	does in fact extend the original Sylvester rank function $\rho$ on $\proj(A)$.
	
	We will now prove that, conversely, any $\infty$-periodic rank function $\rho$ on $\Perf(A)$ is necessarily given by the formula (\ref{eq:extendrank}).	Given an interval $I$ of integers, denote by $X_I\in\Perf(A)$ the brutal truncation of $X$ concentrated in degrees $n\in I$, i.e. $(X_I)_n=X_n$ for $n\in I$, and $(X_I)_n=0$ otherwise, with differential $d_n$ equal to that of $X$ for all $n$ such that $\{n,n-1\}\in I$.
	Since the claimed formula for $\rho(X)$ is consistent with condition (\ref{Op1}), it suffices to prove that the constant term of $\rho(X)$ is equal to
	$\rho(X_0)-\rho(d_0)-\rho(d_1).$

	From the standard exact triangle 
	$$ \Sigma^{-1}X_{[2,\infty)}\stackrel{\phi}{\to}X_{(-\infty,1]} \to X \to X_{[2,\infty)}\rightsquigarrow$$ 
	we obtain by (\ref{Op3}):
	\begin{equation}\label{eq:constterms}
		\rho(X_{(-\infty,1]})-\rho(X)+\rho(X_{[2,\infty)})=(q+1)\rho(\phi).
	\end{equation}
	Next note that by Lemma \ref{lemma:perfrank2}, the constant terms of $\rho(X_{[2,\infty)})$ and of $\rho(\Sigma^{-1}X_{[2,\infty)})$ are zero, and then the constant term of $\rho(\phi)$ is likewise zero since $\rho(\phi)\leq \rho(\Sigma^{-1}X_{[2,\infty)})$ by Lemma~\ref{lem:boundmorphismrank}. It follows from equation (\ref{eq:constterms}) that the constant terms of $\rho(X)$ and $\rho(X_{(-\infty,1]})$ coincide. 
	
	Arguing similarly with the exact triangle 
	$$ \to X_{(-\infty,1]} \to X_{[-1,1]}\to \Sigma X_{(-\infty,-2]}\rightsquigarrow,$$
	we find that the constant terms of $\rho(X_{[-1,1]})$ and $\rho(X_{(-\infty,1]})$ coincide and thus, also coincide with the constant term of $\rho(X)$.
	
	Write $\rho(X_{[-1,1]}) = aq^{-1} + b + cq$,  $\rho(X_{[-1,0]}) = a'q^{-1} + b'$, and
	$\rho(X_{[0,1]}) = b'' + c''q$.  From the two exact triangles
	$$ X_{[-1,0]} \to X_{[-1,1]}\to \Sigma X_1\rightsquigarrow \qquad \text{and} \qquad \Sigma^{-1}X_{-1} \to X_{[-1,1]} \to X_{[0,1]}\rightsquigarrow,$$
	and the triangle inequality (\ref{Op3})	we find that $a=a'$, $c=c''$ and $\rho(X_{[-1,1]})_{q=-1}-\rho(\Sigma X_{-1})_{q=-1}+\rho(X_{[0,1]})_{q=-1}=0$. Furthermore, from the exact triangle 
	\[
	X_0\to X_{[0,1]}\to\Sigma X_1\rightsquigarrow
	\]
	we obtain $\rho(X_{[0,1]})_{q=0}=\rho(X_0)-\rho(d_1)$ and similarly $q\rho(X_{[-1,0]})_{q=0}=\rho(X_{-1})-\rho(d_0)$. We then calculate
	\begin{align*}\label{constantterm}
		b&=\rho(X_{[-1,1]})_{q=-1}+a+c \\
		&= \rho(\Sigma X_{-1})_{q=-1} + \rho(X_{[0,1]})_{q=-1}+a'+c''\\
		&=-\rho(X_{-1})+b''+a'\\
		&=\rho(X_{[0,1]})_{q=0} + (q\rho(X_{[-1,0]}))_{q=0}-\rho(X_{-1})\\
		&=\rho(X_0)-\rho(d_1)+\rho(X_{-1})-\rho(d_0)-\rho(X_{-1})\\
		&=\rho(X_0)-\rho(d_1)-\rho(d_0)
	\end{align*}	
	as desired.
\end{proof}	
\begin{rem}\
	\begin{itemize}
		\item
		Let $A$ be an ordinary ring. Let $d<\infty$. Recall from Subsection \ref{periodic} that $\pi_{\infty,d}:\mathbb{Z}(\infty)\to\mathbb{Z}(d)$ is the natural reduction map. It allows one to assign to any normalized $\infty$-periodic rank function $\rho$ on $\Perf(A)$   the normalized $d$-periodic rank function $\pi_{\infty,d}\circ \rho$ on $\Perf(A)$. This is a injective map with a canonical splitting: given a normalized $d$-periodic rank function $\rho$, restrict it to a Sylvester rank function on $A$, and then take the unique extension to a normalized ($\infty$-periodic) rank function $\hat\rho$ on $\Perf(A)$ provided by Proposition \ref{prop:extendrank}. 
		$$
		\xymatrix{
			\NRk_\infty(\Perf(A)) \ar[rr] \ar[rd]^{\sim}& & \NRk_d(\Perf(A)) \ar[ld] \\
			& \operatorname{Syl}(A)}$$\item
		Suppose $\beta: A\to K$ is a map in the {\em homotopy category} of dg rings from a ring $A$ to a $d$-periodic skew-field $K$. Let $\rho=\rho_\beta$ be the normalized $d$-periodic rank function on $\Perf(A)$ obtained  via pullback along $\beta$  of the unique normalized $d$-periodic rank function on the perfect derived category of  $K$. Then $\hat\rho$ may be described as follows. The homomorphism $H(\beta):A\to K$ obtained by passing to homology algebras factors as the composition of a homomorphism $\beta_0:A\to K_0$ and the inclusion $K_0\hookrightarrow K$, and we have
		$\hat\rho=\rho_{\beta_0}$, the pullback along $\beta_0$ of the unique normalized ($\infty$-periodic) rank function on $\Perf(K_0)$. 
		Moreover $\rho=\pi_{\infty,d}\circ \hat\rho$ if and only if (the homotopy class of maps) $\beta$ is realized as an actual ring homomorphism $A\to K$ (and then, it must necessarily be $H(\beta)$).
	\end{itemize}
\end{rem}
\section{Further properties of rank functions}\label{section:further}
As before, we assume that $\C$ is a triangulated category with a $d$-periodic rank function $\rho$. We will examine the behavior of $\rho$ with respect to functors in or out of $\C$. Given another triangulated category $\C'$ and an exact functor $F:\C'\to\C$, the pullback
$\rho'=F^*(\rho)$, assigning to each morphism $f$ of $\C'$ the rank $\rho'(f)=\rho(F(f))$, is clearly also a $d$-periodic  rank function on $\C'$. The pullback of an integral rank function is integral.

On the other hand, a pushforward of $\rho$ is not always possible; later on we will consider the case of a pushforward along a Verdier quotient. 

Lemma \ref{lem:boundmorphismrank} motivates the following definitions.
\begin{defi}\label{def:full}	
	Let $\rho$ be a $d$-periodic rank function on $\C$. We say that a morphism $f:X\to Y$ in $\C$ is
	\begin{itemize}
		\item {\em left $\rho$-full} if $\rho(f)=\rho(X)$;		
		\item {\em right $\rho$-full} if $\rho(f)=\rho(Y)$; \item {\em $\rho$-full} if it is left full and right full.
	\end{itemize}
\end{defi}

\begin{lem}\label{lem:full}
	Let $\rho$ be a $d$-periodic rank function on $\C$.
	\begin{enumerate}
		\item Let
		$$X\xrightarrow{f} Y\xrightarrow{g} Z \rightsquigarrow$$
		be an exact triangle in $\C$, with a connecting homomorphism $h:\Sigma^{-1}Z\to X$. Then
		\begin{itemize}
			\item $f$ if left $\rho$-full if and only if $g$ is right $\rho$-full if and only if $\rho(h)=0$;
			\item $f$ is $\rho$-full if and only if $\rho(Z)=0$.
		\end{itemize}
		\item Let $f\colon X\to Y$ and $g\colon Y\to Z$ be morphisms in $\C$. If $f$ is $\rho$-full, then $\rho(gf)=\rho(g)$. 
		If $g$ is $\rho$-full, then $\rho(gf)=\rho(f)$.
	\end{enumerate}
\end{lem}
\begin{proof}
	The first part follows directly from Axiom (\ref{M3}). For the second part, consider the exact triangle
	$$\cone(f)\to\cone(gf)\to\cone(g)\rightsquigarrow,$$ given by the octahedral axiom. Suppose that $f$ is $\rho$-full. Then
	$\rho(\cone(f))=0$ and therefore $\rho(\cone(gf))=\rho(\cone(g))$. We deduce that
	$$	(q+1)(\rho(gf)-\rho(g)) = (\rho(X)-\rho(\cone(gf))+q\rho(Z))- (\rho(Y)-\rho(\cone(g))+q\rho(Z)) = 0.$$
	Since, by (\ref{M5}), $\rho(gf)-\rho(f)\geq 0$, we deduce that $\rho(gf)-\rho(g)=0$, as desired. The argument when $g$ is $\rho$-full is similar.
\end{proof}	
\begin{cor}\label{cor:scalar}
	Let $\alpha$ be an invertible element in $\ground$ and $f:X\to Y$ is a morphism in $\C$. Then $\rho(\alpha\cdot f)=\rho(f)$.
\end{cor}
\begin{proof}
	The multiplication by $\alpha$ is an automorphism of $X$ and so, it is a $\rho$-full map. Then the conclusion follows from Lemma \ref{lem:full} (2). 
\end{proof}	
\begin{cor}\label{cor:quotient}
	The full subcategory on
	$$\Ker(\rho):=\left\{X\in\C\colon \rho(X)=0\right\}$$ 
	is a thick subcategory of $\C$. Moreover, $\rho$ descends to a $d$-periodic rank function $\overline{\rho}$ 
	on the Verdier quotient $\C/{\mathcal U}$ by any thick subcategory ${\mathcal U}\subseteq\ker(\rho)$, 
	and the pullback of $\overline{\rho}$ to $\C$ is $\rho$. 
The obtained $d$-periodic rank function on $\C/\ker(\rho)$ is object-faithful.
\end{cor} 
\begin{proof}The category $\C/{\mathcal U}$ is the localization of $\C$ at a set of $\rho$-full morphisms, and so
	any morphism $f:M\to N$ in $C/{\mathcal U}$ can be written as a roof $\xymatrix{M\ar^h[r]&L&N\ar_{g}[l]}$ where $h$ and $g$ are morphisms in $C$ and $\cone(g)$ is an object of $\mathcal U$ (so, in particular, $g$ is $\rho$-full). We set $\overline{\rho}(f):=\rho(h)$; then part (2) of Lemma~\ref{lem:full} ensures that $\overline{\rho}$ is a well-defined function  on morphisms of $\C/{\mathcal U}$. Clearly, $\overline\rho$  
	pulls back to $\rho$, and satisfies (\ref{Mp1}) and (\ref{M2}). That 
	$\overline\rho$ obeys (\ref{M3})-(\ref{M5}) follows from the fact that every triangle or diagram of the form $X\to Y \to Z$ or $X\to Y \leftarrow Z \to W$ in $\C/{\mathcal U}$ is isomorphic to the image of a triangle or diagram of the same form in $\C$.
	
The object-faithfulness of $\overline{\rho}$ on $\C/\ker(\rho)$ is clear.
\end{proof}
\begin{rem}
	Let $d$ and $d^\prime$ be positive integers with $d$ divisible by $d^\prime$; clearly the only element in $\BR(d)$ mapping to zero under the reduction map $\BR_{\geq 0}(d)\to \BR_{\geq 0}(d^\prime)$ is zero. Thus, given a $d$-periodic rank function $\rho$ and the corresponding $d^\prime$-periodic rank function $\rho^\prime$, their kernels coincide. In particular, if one is interested in the thick subcategories that are kernels of rank functions, it results in no loss of generality to consider only ordinary (i.e. 1-periodic) rank functions.
\end{rem}
\begin{lem}\label{lem:sum}Let $f$ and $g$ be morphisms in $\C$ with the same domain and codomain. Then $$\rho(f+g)\leq \rho(f)+ \rho(g).$$\end{lem}
\begin{proof} We have a factorization
	$$(f+g)=
	\begin{pmatrix} 1 & 1 \end{pmatrix}
	\begin{pmatrix} f & 0 \\ 0 & g \end{pmatrix}
	\begin{pmatrix} 1 \\ 1 \end{pmatrix}.$$
	The desired inequality follows from axioms (\ref{M5}) and (\ref{M2}).
\end{proof}
\begin{cor}\label{cor:ideal}
	For any pair of objects $X,Y$ of $\C$, define
	$$\Hom_\C^{\rho}(X,Y)=\left\{f\in\Hom_\C(X,Y)\mid \rho(f)=0\right\}.$$	
	This is an ideal of morphisms in $\C$, by (\ref{M5}) and by Lemma \ref{lem:sum}.
\end{cor}\noproof
\begin{lem}\label{lem:additive}
	Let $f$ and $g$ be morphisms of $\C$ with $\rho(g)=0$. Then $\rho(f+g)=\rho(f)$.
\end{lem}
\begin{proof}
	By Lemma \ref{lem:sum} we have $\rho(f+g)\leq\rho(f)+\rho(g)=\rho(f)$. Taking into account that $\rho(-g)=\rho(g)$ by Corollary \ref{cor:scalar} we have $\rho(f)=\rho(f+g-g)\leq \rho(f+g)+\rho(g)=\rho(f+g)$ which implies the desired equality.
\end{proof}	
\begin{lem}\label{lem:local}
	Let $\rho$ be a prime  rank function on a triangulated category $C$ supplied with a generator $X$ and denote by $L_\rho(X)$ the image of $X$ in the Verdier quotient $\C/\ker{\rho}$. Then the graded algebra $\End_\bullet(L_\rho(X))$ is local.
\end{lem}
\begin{proof}
	Recall that the induced rank function $\overline{\rho}$ on $\C/\ker{\rho}$  object-faithful.  The set $I:=f\in\End(L_{\rho}(X)):\overline{\rho}(f)=0$ is an ideal in $\End_\bullet(L_{\rho}(X))$ by (\ref{M5}). If $f\in\End_\bullet(L_\rho(X))$ is not in $I$ then $\overline{\rho}(f)=1$, in other words $f:L_\rho(X)\to L_\rho(X)$ is a $\overline{\rho}$-full map.  By Lemma \ref{lem:full} the cone of $f$ has zero rank and so is zero in $\C/\ker{\rho}$; thus $f$ is invertible. Since any element in $\End_\bullet(L_\rho(X))$ not in $I$ is invertible, it follows that $I$ is a maximal ideal and $\End_\bullet(L_\rho(X))$ is local.
\end{proof}

So any prime rank function on a triangulated category with a generator $X$ gives rise to a map $\End_\bullet(X)\to F_{\rho}$, where $F_{\rho}$ is the (graded, skew) residue field of     
$\End_\bullet(L_\rho(X))$.

Recall that the \emph{idempotent completion} of an additive category $\C$ is a category $\tilde{\C}$ whose objects
are pairs $(X,e)$ where $X$ is an object of $\C$ and $e$ is an idempotent endomorphism of $X$; a morphism between two such pairs $(X,e)$ and $(Y,t)$ is a morphism $f:X\to Y$ in $\C$ such that $ef=f=ft$. The idempotent completion of a triangulated category is known to be triangulated and it possesses a universal property with respect to exact functors into idempotent complete triangulated categories \cite{BaS01}.

\begin{lem}\label{lem:idempotentcompletion} Any $d$-periodic rank function $\rho\colon  \mathcal{C} \to \R(d)$ on a triangulated category $\mathcal{C}$ extends uniquely
	to a $d$-periodic rank function on its idempotent completion.
\end{lem}
\begin{proof}Given a map $f:(X,e)\to (Y,t)$ in $\tilde{\C}$ we define its rank as the rank of $f:X\to Y$ in $\C$. Note that under this definition the rank of the identity morphism of $(X,e)$ is $\rho(e)$. All axioms of the rank function except (\ref{M3}) are obvious. To check (\ref{M3}), consider an exact triangle in $\tilde{\C}$:
	\[
	(X,e)\stackrel{f}{\to} (Y,t)\stackrel{g}{\to} (Z,k)\rightsquigarrow
	\]
	and note that it is  a direct summand in $\tilde{\C}$ of the	following exact triangle in $\C$
	\begin{equation}\label{eq:dirsum}
		(X,1)\stackrel{f}{\to} (Y,1)\stackrel{g'}{\to} (Z',1)\rightsquigarrow
	\end{equation}
	Moreover, there is the following isomorphism in $\tilde{\C}$:
	\[
	(Z',1)\cong (Z,k)\oplus(\Sigma X,1-e)\oplus (Y,t)
	\]
	and the morphism $g':(Y,1)\to(Z',1)$ can be represented as the composite map
	\[
	\xymatrix{
		(Y,1)\ar^-{\cong}[r]& (Y,t)\oplus (Y,1-t)\ar^-{g\oplus\id_{(Y,1-t)}}[rr]&&(Z,k)\oplus(\Sigma X,1-e)\oplus (Y,1-t)
	}
	\]
	from which it follows that $\rho(g')=\rho(g)+\rho(1-t)$.  The exact triangle (\ref{eq:dirsum}) gives
	\begin{align*}
		\rho(t)+\rho(1-t) &= \rho((Y,t)\oplus(Y,1-t))\\
		&= \rho(Y)\\
		&= \rho(f)+\rho(g')\\ &=\rho(f)+\rho(g)+\rho(1-t)
	\end{align*}
	and it follows that $\rho(t)=\rho(f)+\rho(g)$ as desired.
\end{proof}
Assume now that $\C$ has a generator $X$ and $\rho$ is a prime rank function. An object of $\C$ which is a finite coproduct of shifted copies of $X$ will be called \emph{graded free}. Note that a map between graded free objects could be written as a rectangular matrix $M$ whose entries are graded endomorphisms of $X$, so we can speak of the rank $\rho(M)$ of $M$, generalizing the familiar notion in linear algebra. We will denote by $\overline{M}$ the matrix obtained from $M$ by passing to the graded residue field $F_\rho$ of the graded local algebra $\End_{L_{\rho}}(X)$ and its usual graded rank by $\operatorname{rank}_{F_{\rho}}(\overline{M})$.

\begin{prop} We have $\rho(M) = \operatorname{rank}_{F_{\rho}}(\overline{M})$.
\end{prop}
\begin{proof} We may assume that $\rho$ is object-faithful, so that $\End_\bullet(X)$ is graded local, otherwise replace $\C$ with the Verdier quotient $\C/\Ker(\rho)$. By the usual row reduction process, $M=ER$, where
	$E$ is an invertible matrix and $R$ is a rectangular matrix such that $\overline{R}$ is of the form $\begin{pmatrix} I & 0 \\ 0 & 0 \end{pmatrix}$, where $I$ is an identity matrix. Here $E$ is a product of permutation matrices and upper and lower triangular matrices with invertible elements on the diagonal. So it suffices to prove that $\rho(R) = \operatorname{rank}_{F_{\rho}}(\overline{R})$. This follows from Lemma~\ref{lem:additive}. 
\end{proof}

\begin{cor}\label{cor:rhofull} Let $M$ be a matrix over $\End_\bullet(X)$ representing a morphism between graded free objects. The $M$ has a (square) $\rho$-full submatrix $N$ with $\rho(N)=\rho(M)$. 
\end{cor}
\begin{proof} This follows from the corresponding result for the graded ranks of matrices over a graded skew-field.
\end{proof}

\section{Derived localization of differential graded algebras}\label{section:localization}
Let $A$ be a dg algebra assumed, without loss of generality, to be cofibrant,  and $\tau$ be a thick subcategory of $\Perf(A)$.
\begin{defi}
	A dg $A$-module $M$ is called $\tau$-local if $\RHom_A(X,M)=0$ for any $X\in\tau$. 
\end{defi}
\begin{lem}\label{lem_tensor}
	Let $B$ be a dg algebra supplied with a map $A\to B$. Then $B$ is $\tau$-local if and only if $B\otimes^L_AX\simeq 0$ for any $X\in\tau$. 
\end{lem}
\begin{proof}
Note that  $\RHom_A(X,B)\simeq \RHom_B(X \otimes^L_A B,B)$, thus $\Hom_B(X\otimes^L_A B,B)\simeq 0$ implies $\RHom_A(X,B)\simeq 0$.   Conversely, since $X$ is a perfect dg $A$-module, $X\otimes^L_A B$ is a perfect dg $B$-module. Note that a perfect dg $B$-module (being quasi-isomorphic to its double $B$-dual) is quasi-isomorphic to zero if and only if its $B$-dual is quasi-isomorphic to zero. Therefore
$0\simeq \RHom_A(X,B)\simeq \RHom_B(X\otimes^L_A B,B)$ implies that $X\otimes^L_A B\simeq 0$. 
\end{proof}	
\begin{defi}\label{def:maindef}The derived localization of $A$ with respect to $\tau$ is a dg algebra $L_\tau(A)$ together with a map
$A\to L_\tau(A)$ making it a $\tau$-local $A$-module and such that for any dg algebra $B$ and a map $f:A\to B$ making $B$ a $\tau$-local $A$-module,
there is a unique up to homotopy map $L_\tau(A)\to B$ making the following diagram commutative in the homotopy category of dg algebras:
\begin{equation}\label{eq:homotcomm}\xymatrix{
A\ar^f[r]\ar[d]&B\\
L_\tau(A)\ar@{-->}[ur]
}
\end{equation}
\end{defi}
\begin{rem}\label{rem:derivedinversion}
	Let $s\in A$ be an $n$-cycle of a dg algebra $A$ for $n\in\mathbb Z$, and let $A/s$  be the homotopy cofiber of the left multiplication map $\Sigma^nA\to A$. Then $A/s$ is a perfect $A$-module and the localization with respect to the thick subcategory $\langle A/s\rangle$ generated by it, exists and is given explicitly by the formula $L_sA:=L_{\langle A/s\rangle}(A)= A*_{\ground[s]}^L\ground\langle s,s^{-1}\rangle$, cf. \cite{BCL}. We will now generalize this result to an arbitrary thick subcategory. 
\end{rem}
\begin{theorem}\label{thm:locexists}
For any dg algebra $A$ and any thick subcategory $\tau$ of $\Perf(A)$, the derived localization $L_\tau(A)$ exists and is unique up to homotopy. 
\end{theorem}
\begin{proof}
Let us first suppose that $\tau$ is generated by a single perfect object $X\in\Mod A$, assumed without loss of generality to be cofibrant. Denote by $E$ the dg algebra of endomorphisms of the $A$-module $A\oplus X$. Let $e$ be the element in $E$ given by the projection $E\to A$ along $X$ followed by the inclusion of $A$ into $E$. Then $e$ is a zero-cocycle and an idempotent of $E$. We will show that the desired (derived) localization $L_\tau(A)$ may be constructed as $L_eE$, the (derived) localization of $E$ at $e$, cf. Remark \ref{rem:derivedinversion}.   We have the following commutative diagram of dg categories:
\begin{eqnarray}\label{eq:Morita}\xymatrix {
\Mod A\ar^F[r]&\Mod E\\
\{ X\}\ar[r]\ar@{_{(}->}[u]&\{ E/e\}\ar@{_{(}->}[u]
}
\end{eqnarray}
Here  $F(M)=\Hom_A(A\oplus X,M)$, for $M\in\Mod A$, $E/e$ is the cofiber of the left multiplication by $e$ on $E$, $\{X\}$ the full dg subcategory of $\Mod A$ containing $X$ and closed with respect to arbitrary coproducts, homotopy cofibers and passing to quasi-isomorphic modules and similarly $\{E/e\}$ is the full dg subcategory of $\Mod A$ containing $E/e$, having arbitrary coproducts, homotopy cofibers and closed with respect to passing to quasi-isomorphic modules. Note that $E/e$ is quasi-isomorphic as an $E$-module to $(1-e)E\oplus\Sigma(1-e)E$ and $$F(X)=\Hom_A(A\oplus X,X)\cong X\oplus \Hom_A(X,X)\cong (1-e)X.$$ Therefore the image of $\{X\}$ under $F$ is $\{ E/e\}$ and the commutativity of \ref{eq:Morita} indeed holds.

Since $A\oplus X$ is a compact generator of $\D(A)$, the functor $F$ is a quasi-equivalence, as well as its restriction $\{X\}\to \{E/e\}$.  

Let us construct (a homotopy class of) a map of dg algebras $A\to L_eE$. Note that the functor $F$ could be viewed as tensoring over $A$ with the left $A$-module $\Hom_A(A\oplus X,A)\cong eE$. Composing $F$ with the localization functor into $\Mod L_eE$ we obtain the functor \[G:\Mod A\to\Mod L_eE:M\mapsto M\otimes_A(eE\otimes_EL_eE)\cong M\otimes_A eL_eE.\]
Since $G$ is a dg functor there is an induced (homotopy class of a) map 
of dg algebras \[f:A\cong\End_A(A)\to\End_{L_eE}(eL_eE)\simeq \End_{L_eE}(L_eE)\cong L_eE.\]
The commutative diagram (\ref{eq:Morita}) implies
\[X\otimes_AL_eE\simeq G(X)
\simeq F(X)\otimes_EL_eE\simeq 0,
\]
meaning that $L_eE$ is $\langle X \rangle$-local.

Let us now prove that $L_eE$ has the required universal property with respect to maps from $A$ into $\langle X\rangle$-local algebras.
To this end, let $\tilde{A}:=A\times \ground$. Then $\tilde{A}$ is a dg algebra with the differential $d(a,x)=d_A(a)$  and the product $(a,x)\cdot(
a^\prime,y)=(aa^\prime,xy)$ for $a,a^\prime\in A, x,y\in \ground$.  There is a dg algebra map $\tilde{f}:\tilde{A}\to E$ given by the diagonal embedding $(a,x)\mapsto (l_a,1-e)$ where $l_a$ is the action of $A$ on itself by the left multiplication by $a$. Note that $\tilde{f}$ could also be described as a map 
\[\tilde{A}\cong\End_{\tilde{A}}(\tilde{A})\to\End_{L_eE}(L_eE)\cong E
\]
induced by the functor $-\otimes_{\tilde{A}}E:\Mod{\tilde{A}}\to\Mod E$.

The idempotent $(1,0)\in \tilde{A}$ is mapped to $e\in E$ under $\tilde{f}$, and, slightly abusing notation, we will denote it also by $e$. We obtain an induced map on derived localizations $L_e(\tilde{A})\to L_e(E)$. This produces a homotopy class of dg algebra maps $A\to L_e(E)$ since clearly $L_e(\tilde{A})\simeq A$. This is the same map (up to homotopy) as $f$ as implied by the following diagram of dg categories and dg functors, commutative up to a natural isomorphism.
\[
\xymatrix{
\Mod \tilde{A}\ar^{-\otimes_{\tilde{A}}L_eE}[r]\ar_{-\otimes_{\tilde{A}}A}[d]&\Mod L_eE\\
\Mod A\ar^{-\otimes_AeE}[r]&\Mod E\ar_{-\otimes_EL_eE}[u]
}
\]
Let $B$ be a dg algebra supplied with a map $g:A\to B$ and $\langle X\rangle$-local (so that $X\otimes_A B\simeq 0$); we need to show that $g$ extends uniquely (up to homotopy) to a map $L_\tau(A):=L_eE\to B$. Consider the following diagram of dg algebras, commutative up to homotopy (excluding the dotted arrow).
\begin{equation}\label{eq:curvarrow}
\xymatrix{
\tilde{A}\ar^{\tilde{f}}[r]\ar[d]&E\cong\End_A(A\oplus X)\ar[d]\ar@/^/[ddr]\\
A\simeq L_e\tilde{A}\ar^{L_e(\tilde{f})}[r]\ar^g@/_/[drr]&L_eE=L_{\langle X\rangle}(A)\ar@{-->}[dr]\\
&&B\simeq\End_B(B\oplus X\otimes_A B)
}
\end{equation}
Here the arrow from $E$ to $B$ is induced by the functor $-\otimes_AB:\Mod A\to \Mod B$. By the universal property of the localization $L_eE$ the dotted arrow, making commutative the upper triangle of (\ref{eq:curvarrow}), exists. It follows that the lower triangle of  (\ref{eq:curvarrow}) commutes upon restriction to $\tilde{A}$. However since the map $\tilde{A}\to A\simeq L_eE$ becomes an isomorphism upon inverting $e$, and the dg algebra $B$ is $e$-inverting, we conclude that the lower triangle  of (\ref{eq:curvarrow}) is commutative. Conversely, a similar argument shows that any dotted arrow, making the lower triangle commutative, makes the upper trangle commutative and is, therefore, unique. 

So, the theorem is proved under the assumption that $\tau$ is generated by a single perfect object $X$. It is easy to see that for two perfect objects $X,Y\in \Mod A$, we have \[L_{\langle X\oplus Y\rangle}(A)\simeq L_{\langle X\rangle}(A)*^L_AL_{\langle Y\rangle}(A)\] as both sides satisfy the required universal property with respect to any dg algebra $B$ local with respect to both $X$ and $Y$. Letting $X_s,s\in S$ be a collection of compact generators of $\tau$ indexed by a set $S$, we can set $A_\tau:=\coprod^L_{A, s\in S}L_{\langle X_S\rangle}(A)$, the (derived) coproduct over $A$ of derived localizations $L_{\langle X_S\rangle}(A)$.
\end{proof}
\subsection{Derived localization of \texorpdfstring{$A$}{A}-algebras}
Recall that if 
$C$ is a closed model category and $X$ is an object of $C$ then the 
\emph{undercategory} of $X$ is the category 
$X\downarrow C$ with objects $Y\in C$ supplied with a map $X\to Y$ and morphisms being obvious commutative triangles in $C$. 
The undercategory of $X$ inherits the structure of a closed model category from $C$; in the case when $X$ is cofibrant, this undercategory is homotopy invariant in the sense that for any other cofibrant 
$X^\prime$ and a weak equivalence $X\to X^\prime$ the undercategories $X\downarrow C$ and $X^\prime\downarrow C$ are Quillen equivalent. 
\begin{lem}\label{lem:undercategory}
For a $\tau$-local $A$-algebra $B$ there exists a map of $A$-algebras $L_\tau(A)\to B$, unique up to homotopy in the undercategory  $A\downarrow\Alg$.
 \end{lem}
\begin{proof}
As in the proof of Theorem \ref{thm:locexists} we first consider the case when $\tau$ is generated by a single perfect $A$-module $X$ and arguing similarly with diagram (\ref{eq:curvarrow}), we conclude that there exists a map $g:L_\tau\to B$ making the diagram of dg algebras (which is a fragment of the diagram (\ref{eq:curvarrow}))
\begin{equation}\label{eq:fragment}
\xymatrix{
A\ar[dr]\ar[d]\\
L_\tau(A)\ar^g[r]&B
}
\end{equation}
homotopy commutative in the category of dg algebras. Thus, $g$ could be viewed as a map in the undercategory $A\downarrow\Alg$. The uniqueness of $g$ is proved similarly: suppose that there exists another map of $A$-algebras $g^\prime:L_\tau(A)\to B$ making the (\ref{eq:fragment}) homotopy commutative. Then $g$ and $g^\prime$ are homotopic in $A\downarrow\Alg$ if and only if they are homotopic in $\tilde{A}\downarrow\Alg$ and this, in turn, is equivalent to them being homotopic in $E\downarrow\Alg$. But they are indeed homotopic in $E\downarrow\Alg$ by the defining property of the derived localization $L_eE$, cf. \cite[Definition 3.3]{BCL}. The proof is finished as that of Theorem \ref{thm:locexists}. 
\end{proof}
\begin{rem}
Another way to formulate Lemma \ref{lem:undercategory} is to say that $L_\tau(A)$ is the initial object of the subcategory of the homotopy category of $A\downarrow\Alg$ consisting of $\tau$-local $A$-algebras.	It may appear that this statement and its proof are just rephrasing the corresponding parts of the statement and proof of Theorem \ref{thm:locexists}. The substantive difference is that two maps of $A$-algebras may be homotopic as maps of dg algebras but \emph{not} as maps in the undercategory $A\downarrow\Alg$. We will discuss this discrepancy in more detail below.  
\end{rem}
Let now $B$ be an $A$-algebra, assumed, without loss of generality, to be cofibrant in $A\downarrow \Alg$ (meaning that the given map $A\to B$ is a cofibration of dg algebras). The induction functor $-\otimes_AB:\Mod A\mapsto\Mod B$ takes $\Perf(A)$ into $\Perf(B)$ and (abusing the notation slightly) we will denote by the image of the thick subcategory $\tau\in\Perf(A)$ under this functor, by the same symbol $\tau$.
\begin{prop}\label{prop:undercategory}
We have a natural isomorphism $L_\tau(B)\cong B\ast^L_AL_\tau(A)$ in the homotopy category of $A\downarrow \Alg$.
\end{prop}	
\begin{proof}
This is similar to \cite[Lemma 3.7]{BCL}. There is a Quillen adjunction $A\downarrow\Alg\rightleftarrows B\downarrow\Alg$ with the left adjoint $-\ast_AB$ and the right adjoint being the restriction functor. It is easy to see that this adjunction restricts to an adjunction
between $\tau$-local $A$-algebras and $\tau$-local $B$-algebras and the corresponding homotopy categories. Since left adjoints preserve initial objects, the desired statement follows from Lemma \ref{lem:undercategory}.
\end{proof}
\begin{cor}\label{cor:homotepi1}
	The map of dg algebras $L_\tau(A)\to L_\tau(A)*^L_AL_\tau(A)$ given by the inclusion of the either factor is a quasi-isomorphism.
\end{cor}
Recall from \cite[Chapter 5]{Hov99}, for any closed model category $C$, the notion of a \emph{derived mapping space}, a simplicial set $\Map(X,Y)$ where $X,Y\in C$, generalizing the usual simplicial mapping space in a simplicial model category. Recall also from \cite{Mur16} that a morphism $f:X\to Y$ in a model category is said to be a \emph{homotopy
epimorphism} if for any object $Z$, the induced morphism $\Map(Y,Z)\to\Map(X,Z)$
is an injection on connected components of the corresponding simplicial sets and an
isomorphism on homotopy groups for any choice of a base point. Then we have the following generalization of \cite[Proposition 3.17]{BCL}, which follows, as in op.cit. from Corollary \ref{cor:homotepi1}.
\begin{cor}\label{cor:homotepi2}
	The localization map $A\to L_\tau(A)$ is a homotopy epimorphism.\qed
\end{cor}
	
\subsection{Homotopy coherence} We defined the derived localization $L_\tau(A)$ of a dg algebra $A$ through a certain universal property formulated in the \emph{homotopy category} of dg algebras. It makes sense to ask whether more structured notions, taking into account the $\infty$-structure of the category of dg algebras (i.e. the homotopy type of mapping spaces) can reasonably be considered. We will show that such, apparently more refined, versions of derived localization, nevertheless, turn out to be equivalent to the one defined above. 

\begin{prop}\label{prop:definitions}
	Let $C$ be an $\tau$-local $A$-algebra. The following statements are equivalent:
	\begin{enumerate}
		\item $C$ is isomorphic in the homotopy category of $A\downarrow\Alg$ to $L_\tau(A)$
		\item For any $\tau$-local $A$-algebra $B$ there is a unique map $C\to B$ in the homotopy category of $A\downarrow \Alg$.
		\item For any $\tau$-local $A$-algebra $B$ the map $\Map(L_\tau(A),B)\to\Map(A,B)$ induced by the localization map $A\to L_\tau(A)$ is a weak equivalence.
		\item  For any $\tau$-local $A$-algebra $B$ the mapping space $\Map_A(L_\tau(A),B)$ in $A\downarrow \Alg$ is contractible.
	\end{enumerate} 
\end{prop}
\begin{proof}
There is the following homotopy fiber sequence of simplicial sets:
\begin{equation}\label{eq:fiberseq}
\Map_A(L_\tau(A),B)\to \Map(L_\tau(A),B)\to \Map(A,B);
\end{equation}
here $\Map_A(L_\tau(A),B)$ is the homotopy fiber over the map $A\to B$ that determines $B$ as an $A$-algebra. The second map in (\ref{eq:fiberseq}) is a weak equivalence if and only if its homotopy fiber is contractible over every point. Thus, (3) and (4) are equivalent. The implication $(2)\Rightarrow(1)$ is implied by the following fragment of the long exact sequence of the fibration (\ref{eq:fiberseq}):
\[
\rightarrow\pi_0\Map_A(L_\tau(A),B)\to \pi_0\Map(L_\tau(A),B)\to \pi_0\Map(A,B),
\]
and the reverse implication $(1)\Rightarrow(2)$ is Lemma \ref{lem:undercategory}. The implication $(3)\Rightarrow(1)$ is obvious. Finally, the implication $(1)\Leftarrow(3)$ follows from Corollary \ref{cor:homotepi2}.
\end{proof}
\begin{rem} Let us call a map $B\to C$ in 
	$A\downarrow \Alg$ a $\tau$-local equivalence if  
	for any $\tau$-local $A$-algebra $X$ there 
	is a weak equivalence $\Map_A(C,X)\to\Map_A(B,X)$. Then Proposition \ref{prop:definitions} implies that  for an $A$-algebra $B$ its derived localization $L_\tau(B)\simeq L_\tau(A)*^L_AB$ is the Bousfield localization of $B$ in $A\downarrow\Alg$ with respect to $\tau$-local equivalences, cf. \cite{Hirschhorn03} regarding this notion.
\end{rem}
\subsection{Module localization} We will now relate the notion of derived localization of algebras to the Bousfied localization of $\D(A)$. 
Localization functors exist for a large class of triangulated categories $\C$ and thick subcategories $S$. For example, such is the case when $\C=\D(A)$, the derived category of a dg algebra and $S=\Loc(\tau)$ where $\tau$ is a perfect thick subcategory of $\D(A)$.

The localization of $M\in \Mod A$ with respect to a thick subcategory $\tau\in\Perf(A)$ is a $\tau$-local
$A$-module $N$ together with a map $f:M\to N$ that is a local $\tau$-equivalence, i.e. for any $\tau$-local $A$-module $L$ the induced map
$f^*:\RHom(N,L)\to\RHom(M,L)$ is a quasi-isomorphism. A localization of an $A$-module is clearly defined up to a quasi-isomorphism and, (slightly blurring the distinstion between the category $\Mod A$ and $\D(A)$) we will refer to it as \emph{the} localization of $M$ and denote by $L_\tau^{\Mod A}(M)$.  The following results connect module localization and (derived) algebra localization, generalizing the corresponding results in \cite{BCL}.
\begin{prop}[{\cite[Proposition 2.5]{Dwy06}, \cite[Theorem 4.12]{BCL}}]\label{prop:algmod}
There is a dg algebra $X$ supplied with a dg algebra map $A\to X$ such that $X\simeq L_\tau^{\Mod A}(M)$ as an $A$-module.	
\end{prop}
\begin{proof}
	Let $\tau\otimes 1$ be the thick subcategory in $\Perf(A\otimes A^{op})$ (i.e. perfect $A$-bimodules) generated by $A$-bimodules of the form $X\otimes A^{op}$ with $X\in\tau$. Then the argument of \cite[Theorem 4.12]{BCL} shows that
	$L_\tau^{\Mod A}(M)$ is quasi-isomorphic as an $A$-module to $\REnd_{A^{op}}(L_{\tau\otimes 1}^{\Mod A\otimes A^{op}}(A))$ and the latter is clearly an $A$-algebra.
\end{proof}	
The following result is proved in \cite[Proposition 2.10]{Dwy06}.
\begin{prop}\label{prop:smashing} 
	Let $M$ be an $A$-module. Then $M\simeq M\otimes_A^LA\to M\otimes_A^LL_\tau^{\Mod A}A$
is the localization $L_\tau^{\Mod A}(M)$ of $M$.
\end{prop}\noproof
\begin{cor}\label{cor:algmodloc} The Quillen adjunction $\Mod A\to \Mod L^{\Mod A}_\tau A$ with left adjoint given
	by extension of scalars $M\mapsto M\otimes_A^LL^{\Mod A}_\tau A$
	and right adjoint given by restriction along $A\to L^{\Mod A}_\tau A$
	induces an equivalence between $\D(L^{\Mod A}_\tau A)$ and the full
	subcategory of $\D(A)$ of $\tau$-local modules.
\end{cor}
\begin{proof}
By Proposition \ref{prop:smashing} the functor $M\mapsto M\otimes_A^LL^{\Mod A}_\tau A$ is the $\tau$-localization of the $A$-module $M$; thus if $M$ is already $\tau$-local, then $M\to M\otimes_A^LL^{\Mod A}_\tau A$ is a quasi-isomorphism. Moreover, since $L^{\Mod A}_\tau A$, and $L^{\Mod A}_\tau A$ is $\tau$-local, any $L^{\Mod A}_\tau A$-module is also $\tau$-local as lying in in the localizing subcategory generated by any $L^{\Mod A}_\tau A$. Therefore, for any $L^{\Mod A}_\tau A$-module $M$ the map $M\otimes^L_AL^{\Mod A}_\tau A\to M$ is a quasi-isomorphism.
\end{proof}		
\begin{theorem}\label{thm:algmod} If $L_\tau^{\Mod A}(A)$ is a dg $A$-algebra  which is the localization of $A$
as an A-module then it is also the localization of $A$ as a dg algebra (so that $L_\tau^{\Mod A}(A)$ and $L_\tau(A)$ are isomorphic in the homotopy category of $A$-algebras).	
\end{theorem}
\begin{proof}
This theorem is proved in \cite{BCL} in the special case when $\tau$ is generated by a set of $A$-modules
having the form of a cofiber of an endomorhism of $A$, however the proof continues to hold for arbitrary $\tau$. For the reader's convenience we will repeat the main points. First, we prove that for any $A$-algebra $C$ that is $\tau$-local as an $A$-module there is a (homotopy class of a) map of $A$-algebras $L_\tau^{\Mod A}(A)\to C$. This is Lemma 41.7 of \cite{BCL} and the proof applies verbatim. Since the algebra localization $L_\tau(A)$ is $\tau$-local, there is an $A$-algebra map $f:L_\tau^{\Mod A}(A)\to L_\tau(A)$. Next, by the universal property of $L_\tau(A)$ and since $L_\tau^{\Mod A}(A)$ is a $\tau$-local $A$-algebra, there is a map $g:L_\tau(A)\to L_\tau^{\Mod A}(A)$. The composition $f\circ g:L_\tau(A)\to L_\tau(A)$ is an endomorphism of $L_\tau(A)$ as an $A$-algebra and should, therefore, be homotopic to the identity. Similarly the composition $g\circ f$ is an endomorhism of $L_\tau^{\Mod A}(A)$ that is homotopic to the identity. Thus, $f$ and $g$ are mutually inverse quasi-isomorphisms of dg algebras $L_\tau^{\Mod A}(A)$ and $L_\tau(A)$.
\end{proof}	
\begin{defi}
	A map of dg algebras, $A\to L_\tau(A)$ is called a \emph{finite homological epimorphism} corresponding to a thick subcategory $\tau\in\Perf(A)$.
\end{defi}
\begin{rem}
	A dg algebra map $A\to B$ is called a \emph{homological epimorphism}, \cite{Pau09} if the map
	$B\otimes^L_AB\to B$, induced by the multiplication on $B$, is a quasi-isomorphism. Clearly the derived localization map 
	$A\to L_\tau A$ is a homological epimorphism (since the map $L_\tau(A)\otimes^L_AL_\tau(A)\to A$ is the $\tau$-localization of $L_\tau(A)$ but $L_\tau(A)$ is already $\tau$-local)  but not every homological epimorphism is of this form, owing to the failure of the so-called  \emph{telescope conjecture}, cf. \cite{Keller94a}.
\end{rem}
\begin{cor}\label{cor:thick}
	Let $A$ be a dg algebra. Then there is a 1-1 correspondence between:
	\begin{itemize}
	\item thick subcategories in $\Perf(A)$,
	\item equivalence classes of homotopy classes of finite homological epimorphisms from $A$ where two such  $A\to B$ and $A\to B^\prime$ are equivalent if there is a homotopy commutative diagram 
	\[
	\xymatrix{&A\ar[dl]\ar[dr]\\
		B\ar[rr]&&B^\prime}
	\]
	where $B\to B^\prime$ is a quasi-isomorphism.
	\end{itemize}
\end{cor}
\begin{proof}
	Given a thick subcategory $\tau$ we construct a finite homological epimorphism $A\to L_\tau(A)$. Conversely, associate 
	to a finite homological epimorphism $A\to B$ the kernel of the functor $-\otimes_AB:\Perf(A)\to\Perf(B)$. The universal property of the derived localization $L_\tau(A)$ implies that these constructions are mutually inverse, as claimed.
\end{proof}
\subsection{Derived localization of ordinary rings} Let $A$ be an ordinary algebra and $\tau$ be a thick subcategory in $\Perf(A)$ generated
by a collection of objects represented by complexes of finitely generated projective $A$-modules of \emph{length two}. In other words, the localization map $\Perf(A)\to \Perf(A)/\tau$ inverts some maps between finitely generated projective $A$-modules. If these modules are free, such maps are represented by matrices with entries in $A$ and  derived localization map $A\to L_\tau A$ is the derived version of Cohn's matrix localization \cite{Cohn95}. The more general (still underived) version belongs to Schofield \cite{Sch86}.

\begin{defi}\label{defi:underived}
Let $S$ be a collection of maps between finitely generated projective $A$-modules. An algebra $B$ supplied with a map $A\to B$ is called \emph{$S$-inverting} if the functor $?\mapsto B\otimes_A?$ carries every morphism in $S$ into an isomorphism of $B$-modules. The \emph{localization} of $A$  at $S$ is the $S$-inverting algebra $A[S^{-1}]$ such that for any other $S$-inverting $S$-algebra $B$ the map $A\to B$ factors uniquely through  $A[S^{-1}]$.
\end{defi}
It is clear that $A[S^{-1}]$ is unique if it exists. The existence of $A[S^{-1}]$ is \cite[Theorem 4.1]{Sch86}. 
The following result establishes a precise relationship between the derived and underived notions (and simultaneously gives an independent proof of the existence of $A[S^{-1}]$).
\begin{theorem}
Let $A$, $S$ and $\tau$ be as above. Then $L_\tau A$ is a connective dg algebra i.e. $H_n(L_\tau(A))=0$ for $n<0$ and $H_0(L_\tau(A))=A[S^{-1}]$.
\end{theorem} 
\begin{proof}
By Theorem  \ref{thm:algmod} we can identify $L_\tau(A)$ with $L_\tau^{\Mod A}(A)$. With this, the desired result is the combination of [Proposition 3.1]\cite{Dwy06} and [Proposition 3.2]\cite{Dwy06} (note that op.cit. works with \emph{left} modules but this difference is, of course, unimportant).
\end{proof}
It is natural to ask when Cohn-Schofield localization coincides with derived localization. The following result answers this question.
\begin{prop}\label{prop:stablyflat}
Let $A$, $S$ and $\tau$ be as above. Then the canonical map $L_\tau A\to A[S^{-1}]$ is a quasi-isomorphism if an only if $A[S^{-1}]$ is \emph{stably flat} over $A$ i.e.  $\operatorname{Tor}^A_n(A[S^{-1}], A[S^{-1}])=0$ for $n>1$.
\end{prop}
\begin{proof}
After the identification of $L_\tau(A)$ and $L_\tau^{\Mod A}(A)$, this is proved in [Proposition 3.3]\cite{Dwy06}.
\end{proof}
\begin{cor}\label{cor:hereditary}
Let $A$ be right-hereditary, i.e. having right global dimension $\leq 1$. Then the derived localization $L_\tau$ at any thick subcategory
$\tau$ of $\Perf(A)$ is (quasi-isomorphic to) its Cohn-Schofield localization.
\end{cor}
\begin{proof}
This follows at once from Proposition \ref{prop:stablyflat} since higher Tor functors over a hereditary algebra vanish.
\end{proof}	
\subsection{Commutative rings}	Assume that $A$ is an ordinary commutative ring. Derived localizations corresponding to two-term complexes $A\to A$ given by multiplications by elements of $A$ are ordinary localizations of $A$ as a commutative ring at a multiplicatively closed subset. However for more general thick subcategories $\tau$, it is possible for $L_\tau(A)$ to be a genuine dg algebra. Recall that the well-known Hopkins-Neeman-Thomason theorem \cite{Thomason97} gives the classification of all thick subcategories in $\Perf(A)$: these correspond bijectively to the unions of closed subsets in $\Spec(A)$ having quasi-compact complement (such subsets are sometimes called \emph{Hochster open} sets). Specifically, to any such subset $S$ one associates the thick subcategory of $\Perf(A)$ consisting of perfect $A$-modules having support on $S$. Then one has the following result, in which we $\mathcal A$ stands for the structure sheaf on $\Spec A$ and we identify $\D(A)$ with the category of complexes of quasi-coherent sheaves of $\mathcal A$-modules.
\begin{prop}
	Let $A$ be a Noetherian commutative ring and $\tau_S$ be the thick subcategory of $dg$ $A$-modules with support on a closed subset $S$ of $\Spec A$ and denote by $i:\Spec (A)\setminus S\hookrightarrow\Spec(A)$ the corresponding inclusion map. Then $L_\tau(A)\cong Ri_*i^*(\mathcal A)$ as objects of $\D(A)$.
\end{prop}
\begin{proof}
Let $\D_S(A)$ be the subcategory of $\D(A)$ consisting of complexes of quasi-coherent sheaves on $\Spec(A)$ supported on $S$. It is clear that the (homotopy) fiber of the natural map $\mathcal A\to Ri_*i^*\mathcal A$ is supported on $S$.  On the other hand, the sheaf $Ri_*i^*\mathcal A$ is $\tau$-local; indeed for any $\mathcal A$-module sheaf $\mathcal F$ supported on $S$ we have $i^*\mathcal F\simeq 0$ and so  $\RHom(Ri_*i^*\mathcal A,\mathcal F)\simeq \RHom(i^*(\mathcal A),i^*\mathcal F)\simeq 0$. 

 It follows that $Ri_*i^*\mathcal A$ is the localization of $\mathcal A$ with respect to the localizing subcategory $\D_S(A)$. Now the desired statement follows from Theorem \ref{thm:algmod}.
\end{proof}
\begin{rem}
	The assumptions that $A$ is Noetherian and $S$ is closed are imposed in order for the locally ringed space $\Spec (A)\setminus S$ to be a scheme and for the direct image functor $Ri_*$ to land in the derived category of quasi-coherent sheaves on $\Spec(A)$\footnote{We are grateful to the anonymous referee for pointing out this subtlety to us.}.
\end{rem}
\begin{example}
	Let $A:=\ground[x,y]$ be the polynomial algebra in two variables and $\tau$ be the thick subcategory generated by the
	$1$-dimensional $A$-module $\ground[x,y]/(x,y)$. It is easy to see (e.g. using the Koszul complex) that $L_\tau(A)$ is a dg algebra whose homology is concentrated in degrees $0$ and $-1$.
\end{example}	
It is natural to ask whether for a commutative dg algebra $A$ its derived localization $L_\tau(A)$ is also such. As usual, it is better to consider $E_\infty$ algebras rather than strictly commutative ones. Then a positive answer to this question could be derived by combining the results of \cite{EKMM} and \cite{Man03}. Since $E_\infty$ algebras are rather tangential to the main themes of the present paper, we will only sketch the proof and omit all topological and operadic prerequisites, referring to the two above mentioned sources for details.
\begin{prop}
	If $A$ is a (dg) $E_\infty$ algebra, then $L_\tau(A)$ is also a dg $E_\infty$ algebras and the localization $A\to L_\tau(A)$ can be constructed as a map of dg $E_\infty$ algebras.
\end{prop}
\begin{proof}
Let $H\ground$ be the Eilenberg-MacLane spectrum corresponding to the ring $\ground$; it is known to be a commutative $S$-algebra. According to \cite[Theorem 7.11]{Man03}, there is a functor $\Xi:B\mapsto \Xi(B)$ from the homotopy category of commutative $H\ground$ algebras to the homotopy category of dg  $E_\infty$ algebras and another one $\mathbf{R}:M\mapsto \mathbf{R}(M)$ from the homotopy category of $\Xi(B)$-modules to the homotopy category of $B$-modules. Moreover, both $\Xi$ and $\mathbf R$ are equivalences. Additionally, the $B$-module $\mathbf{R}(\Xi(B))$ is weakly equivalent to $B$ for any $S$-algebra $B$; this property is not stated explicitly in op.cit. but follows readily from the construction. 

It suffices to show that $L_\tau^{\Mod A}$ is quasi-isomorphic to a dg $E_\infty$ algebra and, by the setup described above this is equivalent to showing that the Bousfield localization of $(\Xi)^{-1}(A)$ is a commutative $S$-algebra (and the map into it from $(\Xi)^{-1}(A)$ is that of commutative $S$-algebras). But this is proved in \cite[Chapter 8, Theorem 2.2]{EKMM}.
\end{proof}
\subsection{Small example}\label{smallexample} The following is the smallest example of a finite-dimensional algebra possessing nontrivial derived localization. Let $A$ be the algebra with a basis $e_1,e_2,\alpha_1,\alpha_2$ so that $e_1^2=e_1,e_2^2=e_2, \alpha_1e_1=\alpha_1=e_2\alpha_1, e_1\alpha_2=\alpha_2=\alpha_2e_2$ and the rest of the products are zero. The algebra $A$ is the path algebra of a quiver with two vertices and two arrows between them running in the opposite directions, subject to the relations above. The elements $\alpha_i, i=1,2$ correspond to the two arrows and the elements $e_i,i=1,2$ correspond to the trivial loops at each vertex. 

Set $B:=\REnd_A(\ground,\ground)$ where $A$ acts on $\ground$ via $e_2=\alpha_i=0, i=1,2$. Also set 
$B^\prime:=\REnd_A(\ground\times\ground, \ground\times \ground)$ where  $A$ acts on $\ground\times \ground$ via $\alpha_i=0, i=1,2$. It is 
immediate that $B^\prime$ is (quasi-isomorphic to) the path algebra of the same quiver with arrows marked by the generators
$\alpha^\prime_i, i=1,2$ with $|\alpha^\prime_i|=-1$ and no relations. It follows that $B\simeq e_1B^\prime e_1$ is the polynomial algebra
on one generator $\beta=\alpha_2^\prime\alpha_1^\prime$ with $|\beta|=-2$. It further follows that $\REnd_B(\ground,\ground)\simeq \ground[\alpha]/(\alpha^2)$, the exterior algebra on one generator $\alpha$ with $|\alpha|=1$.

On the other hand, it is clear that $L_{e_1}A$ is the localization of $A$ as an $A$-module with respect to the functor $\RHom_A(-, e_1Ae_1)\simeq\RHom_A(-,\ground)$ and the latter localization is (quasi-isomorphic to) $\REnd_B(\ground,\ground)$, cf. \cite[Theorem 2.1, Proposition 4.8]{Dwyer02} for this kind of statement. 

All told, we conclude that $L_{e_1}A\simeq \ground[\alpha]/(\alpha^2)$. The nonderived localization $A[e_1^{-1}]$  of $A$ is, of course, the ground ring $\ground$.

Next, consider the projective $A$ modules $e_1A$ and $e_2A$; then the left multiplication with $\alpha_2$ determines a map $e_1A\to e_2A$ which we will regard as an object in $\Perf(A)$. Denote by $\tau$ the thick subcategory generated by this object. Then $L_\tau A$ is isomorphic to $M_2(\ground)$, the $2\times 2$ matrix algebra over $\ground$ and since $M_2(\ground)$ is flat over $A$, we conclude that no higher derived terms are present, i.e. $L_{\tau}(A)\simeq M_2(\ground)$. Similar conclusions can be made regarding derived localizations of $A$ at $e_2$ and at the object $e_2A\to e_1A$ determined by the left multiplication by $\alpha_1$. 
\subsection{Group completion} Let $M$ be a discrete monoid and $BM$ be its classifying space. Then the based loop space $\Omega BM$ is the so-called group completion of $M$ and according to McDuff's theorem, any topological space $X$ is weakly equivalent to some $\Omega BM$ \cite{McD79}. The chain algebra $C_*(\Omega BM)$ is quasi-isomorphic to the derived localization of the monoid algebra $\ground[M]$ at all monoid elements by \cite[Theorem 10.3]{BCL}. It follows that the derived localization of an ordinary (ungraded) algebra, such as $\ground[M]$ can be a fairly arbitrary dg algebra. Here is a particularly striking example due to Fiedorowicz \cite{Fie84}.
\begin{example}\label{ex:fiedorowicz}
	Let $M$ be the monoid with  five elements $\{1,x_{ij},i,j=1,2\}$ which multiply according to the rule $x_{ij}x_{kl}=x_{il}$. It is easy to see that the non-derived group completion of $M$ is trivial and that $H_*(M,\ground):=\operatorname{Tor}_*^{\ground[M]}(\ground,\ground)$ coincides with the homology of $S^2$, the two-dimensional sphere. It follows that $BM$ is weakly equivalent to $S^2$ and, therefore $\Omega BM$ is weakly equivalent to $\Omega S^2$. The homology of $\Omega S^2$ is $\ground[x]$ with $|x|=1$, and this dg algebra is clearly formal. Thus, the derived localization of $\ground[M]$ is (quasi-isomorphic to) $\ground[x]$. Note that the localization map $\ground[M]\to C_*(\Omega BM)\simeq \ground[x]$ is highly nontrivial in the homotopy category of dg algebras (e.g. it induces a Verdier quotient on the level of derived categories) and it is \emph{not} the one that factors through $\ground$.
\end{example}
\subsection{Rank functions for derived localization algebras} Let $A$ be a dg algebra and $\rho$ be a rank function on $\Perf(A)$. Consider the map
$A\to L_{\rho}(A)$ from $A$ into its derived localization at $\Ker(\rho)$. Then the following result holds.
\begin{theorem}
The rank function $\rho$ descends to an object-faithful rank function $\overline{\rho}$ on $\Perf(L_\rho(A))$ (so that the pullback of $\overline{\rho}$ under the direct image functor $\Perf(A)\to\Perf(L_\rho(A))$ is $\rho$). 
\end{theorem}
\begin{proof}
By Corollary \ref{cor:quotient}, $\rho$ descends to an object-faithful rank function  on the Verdier quotient $\Perf(A)/\Ker(\rho)$.  Note that $ \Perf(L_\rho(A))$ is the idempotent completion of $\Perf(A)/\Ker(\rho)$ by \cite[Theorem 2.1]{Neema92} and so the obtained rank function on $\Perf(A)/\Ker(\rho)$ extends further to $\overline{\rho}$ on $ \Perf(L_\rho(A))$ by Lemma \ref{lem:idempotentcompletion}. Clearly, $\overline{\rho}$ has the required properties.
\end{proof}
\begin{rem}
The above theorem is an extension of the corresponding statement for Sylvester rank functions, \cite[Theorem 7.4]{Sch86}. This is a key result in theory of Sylvester  rank functions and its proof in op.cit. is very involved. The almost trivial proof of the much more general result above demonstrates the advantage of the notion of a rank function for triangulated categories over the classical notion. 
\end{rem}
\subsection{Loops on p-completions of topological spaces} Another example of derived localization in topology comes from the study of chain algebras of based loops on completed classifying spaces of finite groups, cf. \cite{CohenLevi96, Benson09}.  Let $X$ be a topological space such that $\pi_1(X)$ is finite (e.g. the classifying space of a finite group) and $X^{\wedge}_p$ is the $p$-completion of $X$. Then it is proved in \cite{ChuangLazarev20} that there is an idempotent $e\in\mathbb{F}_p[\pi_1(X)]$  such that the derived localization $L_e\mathbb{F}_p[\pi_1(X)]$ is quasi-isomorphic, as a dg algebra, to $C_*\Omega(X^{\wedge}_p)$, the chain algebra of the based loop space of $X^{\wedge}_p$.
\begin{rem}
When $X$ is the classifying space of a finite group, this result (in a somewhat different formulation) was proved in the paper \cite{Vogel17} where a good portion of derived localization theory was also developed. Unfortunately, the present authors had not been aware of this earlier work and did not make a proper attribution to it in \cite{BCL, ChuangLazarev20}. 
\end{rem}

\section{Localizing rank functions and fraction fields}\label{section:localizing}
We will start with the following almost obvious result.
\begin{prop}\label{prop:localizing} Let $f:X\to Y$ be a morphism in a triangulated category $\C$ supplied with a rank function $\rho$ such that $\rho(f)=0$. Then the following 
conditions are equivalent:
\begin{enumerate}
\item The morphism $f$ factors through an object of rank zero.
\item There exists an object $Z$ in $\C$ and $\rho$-full morphism $g:Z\to X$ such that $f\circ g=0$.
\item There exists and object $W$ in $\C$ and a $\rho$-full morphism $h:Y\to W$ such that $h\circ f=0$.
\item The morphism $f$ maps to zero under the Verdier quotient map $\C\to\C/\ker(\rho)$.
\end{enumerate}
\end{prop}
\begin{proof} The equivalence of (1) with (2) and (3) follows from Lemma \ref{lem:full} and the equivalence of (2) and (3) with (4) follows from the characterization of morphisms in a Verdier quotients in terms of left or right fractions.
\end{proof}
\begin{defi}
An integral rank function  on a triangulated category $\mathcal{C}$ is {\em localizing} if any morphism in $\C$ of rank zero 
satisfies either of the equivalent conditions of Proposition \ref{prop:localizing}. 
\end{defi}
\begin{rem}
The notion of a localizing rank function is motivated by constructing derived localization of dg rings to dg simple Artinian rings. If one is interested in maps into dg algebras more general than dg simple Artinian rings (e.g. dg analogues of von Neumann regular rings), one can speculate that real-valued rank functions will be relevant. In this context perhaps it is more natural to require that any rank 0 morphism factors through objects of arbitrarily small rank. We will not elaborate on this more subtle notion in the present paper however.
\end{rem}

\begin{theorem}\label{thm:category-rank-bijection}
Let $\mathcal{C}$ be a triangulated category admitting a generator. Then there is a bijection between the following two sets:
\begin{itemize}
	\item localizing prime rank functions;
	\item thick subcategories of $\mathcal{C}$ with a simple Verdier quotient.
	\end{itemize}
	\end{theorem}
\begin{proof}
	Let $\tau$  be a thick subcategory of $\mathcal{C}$ such that $\mathcal{C}/\tau$ is simple. Fix an indecomposable object $X$ in $\mathcal{C}/\tau$. Then there is a unique morphism-faithful rank function on $\mathcal{C}/\tau$ taking value $1$ on $X$, see Corollary \ref{cor:simplerank}.  The pullback of this rank function to $\mathcal{C}$ is a localizing rank function on $\mathcal{C}$.
	
	Conversely, given a  localizing rank function $\rho\colon \mathcal{C} \to \Z$, let $\overline{\rho}$
 be the rank function on $\C/\Ker(\rho)$ induced by $\rho$; since $\rho$ is localizing, $\overline{\rho}$ is morphism faithful and $\overline{\rho}(X)=1$ for some generator $X$. By Proposition \ref{prop:faithfulsimple}, $\C/\Ker(\rho)$ is simple.
	
	It is clear that the two processes described define mutually inverse maps between the two sets in the statement of the theorem. 
	\end{proof}
The following result gives a complete description of homotopy classes of derived localizations of dg algebras into dg skew-fields or, more generally, dg simple Artinian rings, in terms of rank functions. Classically, only partial results of this sort were available (e.g. for a very specific class of rings or a particular class of localizations), cf. \cite[Theorems 5.4, 5.5]{Sch86} and \cite[Theorem 4.6.14]{Cohn95}.

\begin{theorem}\label{thm:algebra-rank-bijection}
	Let $A$ be a dg algebra. Then there is a bijection between the following two sets:
	\begin{itemize}
	\item  localizing prime rank functions on $\Perf(A)$.
	\item equivalence classes of homotopy classes of finite homological epimorphisms $A\to B$ into simple Artinian dg algebras $B$.
	\end{itemize} 
Moreover, $\rho(A)=1$ if and only if $B$ is a dg skew-field.
	\end{theorem}
\begin{proof}
By Theorem \ref{thm:category-rank-bijection}  localizing prime rank functions $\rho$ on $\Perf(A)$ correspond bijectively to thick subcategories on $\Perf(A)$ with a simple Verdier quotient. For such a thick subcategory $\tau$ the image of $A$ in $\Perf(A)$ is a generator $M$ of $\Perf(A)/\tau$ that we can assume without loss of generality to be cofibrant over $A$. Since $ \Perf(L_\tau(A))$ is the idempotent completion of $\Perf(A)/\tau$ and $\Perf(A)/\tau$ is simple, $ \Perf(L_\tau(A))$ is also simple and it follows by Proposition \ref{prop:simple} that  $L_\tau(A)$ is a simple Artinian dg algebra. By Corollary \ref{cor:thick} such thick subcategories correspond bijectively with the equivalence classes of homotopy classes of finite homological epimorphisms $A\to L_\tau(A)$. Finally, the condition $\rho(A)=1$ means that $M$ is an indecomposable object of $\Perf(A)/\tau$ and in this case $H(L_\tau(A))\cong \End_{\Perf(A)/\tau}(M,M)$ must be a graded skew-field.
	\end{proof}
\begin{example}
Let $A$ be the 4-dimensional algebra of \S\ref{smallexample}. The localization map $A\to L_{\tau}(A)\simeq M_2(\ground)$ induces a functor $\Perf(A)\to\Perf(M_2(\ground))$ where the target triangulated category is simple. The induced rank function on $\Perf(A)$ is localizing. 

On the other hand, consider the localization map 
$A\to L_{e_1}(A)\simeq \ground[\alpha]/\alpha^2$. 
The category $\Perf(\ground[\alpha]/\alpha^2)$ is not simple, but composing with the augmentation $\ground[\alpha]/\alpha^2\to\ground$, we obtain
a map $A\to \ground$. Note that the latter map is the Cohn-Schofield (nonderived) localization of $A$. The induced rank function (which corresponds to a certain Sylvester rank function for $A$) is \emph{not} localizing.
\end{example}
\subsection{Localizing rank functions for hereditary rings}
Localizing rank functions are easiest to construct for perfect derived categories of (right)-hereditary algebras in which case 
they essentially reduce to the nonderived notion.

Let $A$ be an ordinary $\ground$-algebra and $K_0(A)$ be its Grothendieck group of its category of finitely generated projective modules. The abelian group $K_0(A)$ has a pre-order specified by the declaring the classes of finitely generated projective modules in $K_0(A)$ to be positive. Then we have the notion of a projective rank function on $A$ cf. \cite{Sch86}. 
\begin{defi}
	A projective rank function on $A$ is homomorphism of pre-ordered groups $\rho:K_0(A)\to\R$ for which $\rho[A]=1$.
\end{defi}
A Sylvester rank function on $A$ (which, by Theorem \ref{thm:Sylvesterderived} is equivalent to a normalized $\infty$-periodic rank function on $\Perf(A)$) restricts to a function on the positive cone of $K_0(A)$ with values in nonnegative real numbers and thus, determines a projective rank function on $A$. One can ask whether, conversely, one can associate to a projective rank function on $A$ a Sylvester rank function on $A$. For this, a rank needs to be assigned to any map between two finitely generated projective $A$-modules. One can attempt the following definition.
\begin{defi}
Let $A$ be a ring with a projective rank function $\rho$.	Given a map $f:P\to Q$ between two finitely generated projective $A$-modules,  its \emph{inner} rank $\rho(f)$ is defined as $\rho(f):=\inf(\rho(S))$ where $S$ ranges through finitely generated projective $A$ modules through which $f$ factors.
\end{defi}
There is no reason for an inner rank function to be Sylvester in general. However, this holds in one important special case.
\begin{prop}\label{prop:hereditary}
	Let $A$ be a hereditary algebra with a projective rank function $\rho$. Then the associated inner rank function is Sylvester.
\end{prop}
\begin{proof}
We prove the desired statement by quoting relevant results of \cite{Sch86} (all of which are elementary and easy). Firstly, by \cite[Theorem 1.11]{Sch86} the Sylvester law of nullity holds for $\rho$; that is if $\alpha:P_0\to P_1$ and $\beta:P_1\to P_2$ are  two maps between finitely generated projective $A$-modules for which $\beta\circ \alpha=0$ then $\rho(\alpha)+\rho(\beta)\leq \rho(P_1)$.

Next, the law of nullity implies axioms (\ref{m2}) and (\ref{m3}) of the Sylvester morphism rank functions, by \cite[Lemmata 1.14, 1.15]{Sch86}. Finally, the axioms (\ref{m1}) and (\ref{m4}) are immediate from the definition.
\end{proof}
\begin{rem}
As is clear from the above proof, Proposition \ref{prop:hereditary} is essentially contained in \cite{Sch86}, although it is not explicitly formulated in op. cit. as such. Moreover, the statement of the proposition holds under the weaker assumption that $A$ be weakly semihereditary (as opposed to hereditary). We will not need this stronger result.	
\end{rem}
\begin{cor}
	Let $A$ be a hereditary algebra with a projective rank function $\rho$. Then $\rho$ extends uniquely to a rank function on $\Perf(A)$.
\end{cor}
\begin{proof}
The inner rank function associated to $\rho$ is Sylvester by Proposition \ref{prop:hereditary}. Next, Proposition \ref{prop:extendrank} implies that a Sylvester rank function gives rise to an $\infty$-periodic rank function on $\Perf(A)$ and the latter determines, by reduction, a (1-periodic) rank function on $\Perf(A)$.
\end{proof}	
\begin{rem}
For a hereditary algebra $A$ there is a 1-1 correspondence between rank functions and $\infty$-periodic rank functions on $\Perf(A)$. Indeed, this is straightforward to check on complexes of finitely generated projective $A$-modules of length 2 using formula (\ref{eq:extendrank}) and the hereditary property implies that such complexes generate the whole triangulated category $\Perf(A)$.
\end{rem}
\begin{prop}
Let $A$ be a hereditary algebra and $\rho$ be a rank function on $\Perf(A)$ associated, as above, with a projective rank function on $A$. Then $\rho$ is localizing.
\begin{proof}The rank function $\rho$ is localizing if and only if the derived localization $L_{\rho}A$ of $A$ at $\Ker\rho$ is a dg skew-field by Theorem \ref{thm:algebra-rank-bijection} and by Corollary \ref{cor:hereditary} $L_{\rho}A$ is quasi-isomorphic to the underived localization of $A$ at the collection of $\rho$-full maps between finitely generated $A$-modules. By \cite[Theorem 5.4]{Sch86} this underived localization is a skew-field and we are done.
\end{proof} 
\end{prop}
A standard application of this result is the construction of the skew-field of fractions of a free algebra. 
\begin{example}
Let $\ground\langle S\rangle$, the free algebra on a set $S$. Then there exists a unique homological epimorphism $A\to K(S)$ where $K(S)$ is a skew-field. Indeed, $A$ is a hereditary algebra with $K_0(A)=\BZ$, thus there exists a unique projective rank on $A$ that uniquely extends to a rank $\rho$ function on $\Perf(A)$. The corresponding derived localization $L_\rho(A)$ (which coincides with the Cohn-Schofield localization since $A$ is hereditary) is therefore a skew-field, commonly known as the \emph{free field} on $S$, \cite{Cohn95}.
\end{example}
\subsection{Derived fields of fractions}
Given an ordinary ring $A$ we will understand a \emph{classical} field of fractions of $A$ to be a (nonderived) localization  $A\to K$ where $K$ is a skew-field. A derived field of fractions of $A$ is a dg skew-field $K$ together with a derived localization map $A\to K$.
\begin{rem}
This definition of a classical field of fractions agrees with the standard notion that ordinarily assumes that $A$ is a commutative domain. In the noncommutative case the definition accepted e.g. in \cite{Cohn95} is different, in particular the map $A\to K$ is assumed to be an embedding. Our definition is one that extends most naturally to the dg context.
\end{rem}
Corollary \ref{thm:algebra-rank-bijection} classifies derived fields of fractions of a ring $A$ in terms of localizing rank functions on $\Perf(A)$.  The following examples demonstrate the existence of such (genuinely derived) fields of fractions and thus, localizing rank functions (that cannot be reduced to classical Sylvester rank functions) even for finite-dimensional algebras over fields.
\begin{example}\
\begin{enumerate}\item
Let $A=\ground[M]$ be the 5-dimensional algebra of Example \ref{ex:fiedorowicz}. We saw that there is a derived localization map
$\ground[M]\to \ground[x]$ where $|x|=1$. Composing this with inverting $x$, we obtain a derived localization map $\ground[M]\to \ground[x, x^{-1}]$ where the target is a graded skew-field of period 1, i.e. a derived fraction field of $A$. The `classical' part of this map is the augmentation map $\ground[M]\to \ground$; this is an underived localization and thus, a classical fraction field in the sense understood above.
\item
This example is essentially contained in \cite[Example 6.13]{ChuangLazarev20}.	Let $\ground$ be a field of characteristic $3$ and consider $A:=\ground[S_3]$, the group algebra of the symmetric group on 3 symbols. Then $A$ contains an idempotent $e$ such that $L_eA$ is (quasi-isomorphic to) the graded algebra generated by two indeterminates $y$ and $z$ with $|z|=2,|y|=3$ and $z^3=y^2$. The latter algebra has a fraction field $k[x, x^{-1}]$ with $x=z^{-1}y$ so $|x|=1$ . Thus, $k[x, x^{-1}]$ is a derived fraction skew-field for $A$.  Again, the `classical' part of the latter derived fraction field is the augmentation map of the group ring $\ground[S_3]$.
\end{enumerate}
\end{example}

\bibliography{biblibrary}
\bibliographystyle{alphaurl}

\end{document}